\documentclass[11pt]{amsart}
\usepackage{amsfonts,amssymb,amscd}
\usepackage{oldgerm}
\usepackage{amsfonts}
\usepackage{amssymb}
\usepackage{graphics}
\usepackage{booktabs}

\usepackage{booktabs, tabularx}
\usepackage[margin=0.8in]{geometry}

\usepackage{xfrac}

\usepackage{ucs}
\usepackage[utf8x]{inputenc}

\usepackage{epic,eepic,pstricks,graphics,color}

\newcommand{\tpp}{{\tt p}}

\newcommand{\C}{\mathbb{C}}

\newcommand{\N}{\mathbb{N}}
\newcommand{\Z}{\mathbb{Z}}

\newcommand{\fol}{\mathcal{F}}
\newcommand{\tilf}{\widetilde{\mathcal{F}}}

\newcommand{\partx}{{\partial /\partial x}}
\newcommand{\party}{{\partial /\partial y}}
\newcommand{\partz}{{\partial /\partial z}}
\newcommand{\calp}{{\mathcal{P}}}
\newcommand{\ddx}{{\partial /\partial x}}
\newcommand{\ddy}{{\partial /\partial y}}

\def\picill#1by#2(#3)#4
{\vbox to #2
{\hrule width #1 height 0pt depth 0pt
\vfill\special{illustration #3 scaled #4}}}

\newtheorem{teo}{Theorem}[section]
\newtheorem{prop}[teo]{Proposition}
\newtheorem{lema}[teo]{Lemma}
\newtheorem{obs}[teo]{Remark}

\newtheorem{coro}[teo]{Corollary}

\begin{document}

\title{Complete commuting vector fields and their singular points in dimension~$2$}

\author{Ana Cristina Ferreira, \, \, \, Julio C. Rebelo \, \, \, \& \, \, \, Helena Reis}
\address{}
\thanks{}


\date{\today}

\maketitle

{\centerline{\small{\it Dedicated to Felipe Cano on the occasion of his $60^{\rm th}$-birthday}}}

\maketitle

\begin{abstract}
We classify degenerate singular points of $\C^2$-actions on complex surfaces.
\end{abstract}

\section{Introduction}

The notion of {\it germ of semicomplete vector field}\, was introduced in \cite{JR1} essentially as
an obstruction to the realization problem for the given germ by an actual {\it complete}\, holomorphic vector field
on some suitable manifold. In other words, if a germ of vector field {\it is not}\, semicomplete, then it
{\it cannot}\, be realized as a singular point of a complete vector field on any complex manifold. In dimension~$2$,
(germs of) semicomplete holomorphic vector fields were classified in \cite{JR1}, \cite{JR}, and \cite{JR2}.
Apart from singular points where the vector field has at least one eigenvalue different from zero, this classification
is summarized by Table~\ref{TheTable} in Section~\ref{morelikebasics}.

%
%

Although the mentioned classification provides very accurate normal forms, a few interesting questions
can still be formulated. We may ask for example if the ``invertible'' multiplicative function appearing in several
entries of the mentioned table can be made constant. The answer to this question is in general negative, as already
follows from \cite{JR}, but additional information is provided in Section~\ref{globalexamples} and  especially by Proposition~\ref{non_invertible_function}. Similarly, still
considering Table~\ref{TheTable}, we note the presence of vector fields with non-isolated singularities
whose underlying foliation is not linearizable (item~10 in Table~\ref{TheTable}).
A first motivation to this work was then to consider pairs of commuting (semicomplete) vector fields and
check the extent to which
the existence of this additional symmetry further rigidifies the corresponding germs. This motivation is compounded
by the observation that most of the models in Table~\ref{TheTable} possess non-constant first integrals since
the existence of non-trivial centralizers is a typical property of integrable systems.

Whereas some additional motivations arising from singular spaces and/or commuting vector fields in dimension~$3$
will also be mentioned below, let us first explain the main results obtained in this paper. Our main purpose is
to classify all pairs $X$, $Y$ forming a {\it rank~$2$}\, system of commuting semicomplete vector fields on a neighborhood
of $(0,0) \in \C^2$, where by {\it rank~$2$ system} it is meant that $X$ and $Y$ {\it are not}\, linearly dependent
everywhere. In the sequel, the search for
this classification will be referred to as the {\it centralizer problem}. Note that the above mentioned invertible multiplicative
function appearing in front of the models in Table~\ref{TheTable} significantly adds to the difficulty of
the centralizer problem as it can be inferred from a simple computation of the bracket of two vector fields.
The purpose of this paper
is to provide an accurate solution to the centralizer problem whose content is summarized by Main Theorem below.

Before stating our classification result, it is however convenient
to introduce some terminology. One of the most interesting cases appearing in Main
Theorem, as well as in Table~\ref{TheTable}, involves vector fields whose underlying foliations
are determined by a holomorphic $1$-form $\omega$ admitting the normal form
\begin{equation}
\omega = mx (1 + {\rm h.o.t.}) \, dy \; + \; ny (1 + {\rm h.o.t.}) \, dx \label{general_martinetramis_form}
\end{equation}
where ${\rm h.o.t.}$ stands for higher order terms and with strictly positive integers $m,n$.
In the sequel we will also use the phrase {\it Martinet-Ramis foliations}\,
to refer to the collection of foliations described in Equation~(\ref{general_martinetramis_form}) since they are thoroughly studied
in~\cite{MR-ENS}. Table~\ref{TheTable} tells us in particular that there are non-linearizable
Martinet-Ramis foliations associated with semicomplete vector fields having curves of singular points. Furthermore
these foliations admit only constant first integrals.

The main result of this paper is therefore the following one:

\vspace{0.1cm}

\noindent {\bf Main Theorem.} {\sl Suppose that $X, \, Y$ form a rank~$2$ system of commuting semicomplete holomorphic
vector fields on $(\C^2, 0)$. Assume also that the eigenvalues of both $X$ and $Y$ at the origin are zero.
Then there are coordinates $(x,y)$ where the pair $X, \, Y$ admits one of the following normal forms
(up to linear combination):
:
\begin{itemize}
   \item[($\imath$)] $X = y^n \ddx$ and $Y = y (a(x,y) \ddx + b(y) \ddy)$ where $b(y) = \alpha y + {\rm h.o.t.}$ with
    $\alpha \in \C^{\ast}$ and $a(x,y) =1 + \alpha n x +xr(y) +s(y)$ where $r, s$ are holomorphic functions
    satisfying $r(0) =s(0) =0$.

   \item[($\imath \imath$)] $X = x^2 \ddx$ and $Y = y^2 \ddy$.

   \item[($\imath \imath \imath$)] $X = y^n \partial /\partial x$ and $Y = y(nx \partial /\partial x + y \partial /\partial y)$
   where $n \in \N^{\ast}$, $n \geq 2$. If $n=1$, then we have
   $$X =  c_1 y \ddx + c_2 x(x \partial /\partial x + y \partial /\partial y)  \; \; \;
   {\rm and} \; \; \; Y = y(x \partial /\partial x + y \partial /\partial y)
   $$
   with $c_1, c_2 \in \C$.

   \item[($\imath v$)]
   $$
    X = g_1 (xy^n) y^nx^2 \left[ \frac{\partial}{\partial x} + y^{n+1} g_2 (xy^n)  \left( nx \frac{\partial}{\partial x}
    - y \frac{\partial}{\partial y} \right) \right]
   $$
    and $Y = y(nx \ddx +y\ddy)$, where $n \in \N^{\ast}$ and $g_1$, $g_2$ are holomorphic functions of a single variable
    with $g_1 (0) =1$ and $g_1'(0) = g_2 (0) =0$;

   \item[($v$)] $X = y^n x^2 \partial /\partial x$ and $Y = x (ny-(n+1)x) \ddx - y^2\ddy$
    where $n\in \mathbb{N}$.

   \item[($v \imath $)] $X = (y-2x^2) \ddx - 2xy \ddy$ and $Y = y(x \partial /\partial x + y \partial /\partial y)$.

   \item[($v \imath \imath$)] Both $X$ and $Y$ have linearizable Martinet-Ramis singularities at the origin and their
   associated foliations have the same eigenvalues $m, -n$ ($m,n \in \N^{\ast}$). Moreover there are non-negative integers
   $a,b, a_{\mu}, b_{\mu}$ satisfying:
   \begin{enumerate}

   \item $am-bn = \pm 1$ and $(a,b) = k_1 (m,n) + (a_{\mu}, b_{\mu})$, $k_1 \in \N$.

   \item the function $(x,y) \mapsto x^{a_{\mu}} y^{b_{\mu}} / x^ny^m$ is strictly meromorphic.

   \end{enumerate}
   With the preceding notation, $X$ and $Y$ takes respectively on the forms
   \begin{eqnarray*}
   X & = & x^a y^b [mx \partial /\partial x - ny  \partial /\partial y] \, , \\
   Y & = & x^a y^b u_1(x^ny^m)
   [mx \partial /\partial x - ny  \partial /\partial y + x^{-a_{\mu}} y^{-b_{\mu}} u_2 (x^ny^m) [b x \ddx - ay \ddy]]
   \end{eqnarray*}
   where $u_1, u_2$ are holomorphic functions of a single variable, with $u_1(0) =1$, and $u_2$ such that the map
   $(x,y) \mapsto x^{-a_{\mu}} y^{-b_{\mu}} u_2 (x^ny^m)$ is holomorphic of order at least~$1$ at the origin.

\end{itemize}

\noindent Conversely all of the above models yield rank~$2$ systems of commuting semicomplete vector fields.}

\vspace{0.1cm}

\begin{obs}
{\rm In the above list the reader will note that Cases~($\imath \imath \imath$) and~($v \imath$) have a common core
when $n=1$
though it seems to be preferable to state the classification in the above form as opposed to try to merge these two
cases into a single one. The issue relating these two items arises from
the general considerations at the beginning of Section~\ref{morelikebasics}, cf. in particular
Equation~(\ref{forcommutation}). In fact, for $n=1$, the vector field
$y (x \partial /\partial x + y \partial /\partial y)$ admits $x/y$ as first integral and, once multiplied
by $x/y$, yields the vector field $x (x \partial /\partial x + y \partial /\partial y)$ which is still holomorphic.
In particular, the vector fields $x (x \partial /\partial x + y \partial /\partial y)$ and
$y (x \partial /\partial x + y \partial /\partial y)$ must commute (Equation~(\ref{forcommutation})).
Now, just note that the vector fields $x(x \partial /\partial x + y \partial /\partial y)$,
$(y-2x^2) \ddx - 2xy \ddy$, and $y \partial /\partial x$ are linearly dependent over $\C$ and all of them
commute with $y (x \partial /\partial x + y \partial /\partial y)$. Naturally the vector fields
$x (x \partial /\partial x + y \partial /\partial y)$ and $y (x \partial /\partial x + y \partial /\partial y)$ do not
form a rank~$2$ commutative systems and this is the reason why they do not appear together in the above list.

The alternative in Case~($\imath \imath \imath$) can similarly be explained. For $n=1$, both vector fields
$y \partial /\partial x$ and $x (x \partial /\partial x + y \partial /\partial y)$ commute with
$y (x \partial /\partial x + y \partial /\partial y)$. As a side note, for arbitrary $n\in \N^{\ast}$,
the vector field $y(nx \partial /\partial x + y \partial /\partial y)$ admits $x/y^n$ as first integral
though the product of $y (x \partial /\partial x + y \partial /\partial y)$ by $x/y^n$ is not holomorphic unless
$n=1$. Yet, $x/y^n$ times $y^n \partial /\partial x$ yields the vector field $x\partial /\partial x$ which, albeit
holomorphic, possesses one eigenvalue different from zero at the origin and hence is excluded from
the classification in Main Theorem.

Also Case~($\imath v$) with $g_1$ constant equal to~$1$ and $g_2$ identically zero
and Case~($v$) share a common nature. Indeed, $(n+1)x^2 \ddx$
commutes with $y^n x^2 \ddx$ for the evident reasons (Equation~(\ref{forcommutation})) while
the sum of $x (ny-(n+1)x) \ddx - y^2\ddy$ and $(n+1)x^2 \ddx$ yields $y (n x \ddx - y\ddy)$.

A last needed comment about the classification above still concerns Case~($\imath \imath \imath$) when $n=1$.
Precisely, according to Main Theorem, for every choice of coefficients $c_1, c_2$, the vector field
$$
X =  c_1 y \ddx + c_2 x(x \partial /\partial x + y \partial /\partial y)
$$
is semicomplete. This assertion however can easily be justified, in fact, it suffices to show that the vector fields
$y \ddx + c x(x \partial /\partial x + y \partial /\partial y)$ are semicomplete for every $c \neq 0$. By blowing-up
these vector fields, the last assertion turns out to be equivalent to showing that all the vector fields of the form
$$
x^2 \partial /\partial x - (xy + cy^2) \partial /\partial y
$$
are semicomplete. For this, it suffices to note that the linear change of coordinates $(u,v) \mapsto (-cu/2, v) = (x,y)$
conjugates these vector fields to the ``parabolic quadratic vector field'' $x^2 \partial /\partial x - y (x -2y)
\partial /\partial y$ of Table~\ref{TheTable} ($n=1$).}
\end{obs}

An immediate consequence of Main Theorem is that only linearizable Martinet-Ramis singularities exist in the context
of rank~$2$ commuting vector fields. It also shows that
the invertible function appearing in front of several vector fields in Table~\ref{TheTable} can now be made constant
though, in this respect, Proposition~\ref{non_invertible_function} already sharpens Table~\ref{TheTable} to a certain extent.
On the other hand, the statement of Main Theorem suggests that Martinet-Ramis singularities are somehow special which
hardly come as a surprise since they are the basic building blocks of more general singularities in
Table~\ref{TheTable} (see \cite{JR1}, \cite{JR}, \cite{JR2} or the more general ``birational point of view'' of \cite{GuR}).

As a side notice about Main Theorem and some possible extensions, we remind the reader that the study of
pairs of commutative semicomplete germs of vector fields {\it in dimension~$3$} is of particular importance, not least
because known examples include Lins-Neto's examples for the Painlev\'e problem,
(see \cite{alcides} and \cite{guillotCRAS}), as well as several equations in Chazy's list,
see \cite{adolfo-chazy}. These pairs of vector fields defined on $(\C^3, 0)$ however always leave invariant an
(singular) analytic surface through the origin, owing to the main result of \cite{JR-HR}. As matter of fact,
the restriction of the vector fields in question to this invariant surface provides significant
information on the initial action and, in turn, hints at the interest of extending Main Theorem to singular surfaces
along the lines of the birational theory of semicomplete vector fields in \cite{GuR} and \cite{adolfo-kato}.
Finally, from a more technical standpoint, the problem of understanding the structure of commuting vector fields - regardless
of whether or not they are semicomplete - has a number of potential applications. Furthermore, by arguing as in
\cite{GuR} and relying on Seidenberg's theorem \cite{seiden}, in most cases this local problem will only involve
vector fields associated with Martinet-Ramis singularities.
In this sense the method used in Section~\ref{resonantfoliations} to deal with
vector fields associated with Martinet-Ramis singularities may provide some insight on the corresponding vector fields
beyond the scope of the semicomplete situation emphasized in this work.

Another interesting problem that unfortunately will mostly be left out of this work concerns what may be called
the {\it realization problem}. It consists of realizing the semicomplete models provided by Main Theorem
(or even by Table~\ref{TheTable} in the case of single vector field) as a singular point of a complete flow
on some complex surface, not necessarily compact. If we restrict ourselves to {\it compact surfaces}\, then
all models that are realizable are known as a consequence of the classification theorem of Dloussky, Oeljeklaus and Toma
\cite{Dloussky-1}, \cite{Dloussky-2}; see also \cite{adolfo-kato}. In the general case some additional information
is known in connection with ``elliptic models'' in Table~\ref{TheTable} and elliptic surfaces, cf. \cite{JR},
and this topic will further be developed in Section~\ref{globalexamples}. There is however no doubt that the
realization problem deserves a more detailed treatment than the one provide in this paper.

Let us finish this introduction by briefly outlining the structure of the paper.
Section~\ref{morelikebasics} contains a number of basic facts that will be used in the course of this article.
The highlight of this section is Table~\ref{TheTable} summarizing the classification of germs of semicomplete
vector fields in dimension~$2$.

Section~\ref{globalexamples} is primarily devoted to the centralizer problem involving the ``elliptic vector fields''
in Table~\ref{TheTable}. The method used there relies on the realization of the corresponding germs as complete
vector fields on open elliptic surfaces as pointed out in \cite{JR}. The main result of the section is
Proposition~\ref{solving-centralizer_elliptic} explaining why these vector fields have no place in Main Theorem.
Collateral results in this section show, in particular, that elliptic vector fields {\it do admit}\, non-trivial centralizers
in the {\it meromorphic category}. A natural globalization of the resulting pair of meromorphic vector fields is
also indicated in the section which finishes with a similar - though shorter - discussion of Hirzebruch surfaces and
the realization problem involving parabolic vector fields.

Section~\ref{invertiblefunctions_etc} contains Proposition~\ref{non_invertible_function} which, plainly put, improves
Table~\ref{TheTable} in regard to the invertible functions ``$f$'' appearing in this table. There follows from
Proposition~\ref{non_invertible_function} that the function ``$f$'' can be assumed constant in certain cases and this,
along with previously established results, allows us to reduce the proof of Main Theorem to the cases in which
the associated foliations are either regular or of Martinet-Ramis type. The section then finishes by completing this
discussion in the case where one of the foliations is regular. Finally Section~\ref{resonantfoliations} is entirely
devoted to finishing off the proof of Main Theorem by studying the case in which the foliations associated with the
pair of vector fields $X$ and $Y$ are of Martinet-Ramis type.

\section{Basic issues of local nature}\label{morelikebasics}

This section contains a review of the material needed in the course of this paper. Whereas a good part of this material
revolves around semicomplete vector fields and ``maximal local actions'', we will begin with a couple of elementary observations
about commuting vector fields which will be used throughout this paper.

Assume then that $X$ and $Y$ are holomorphic vector fields defined on a neighborhood of $(0,0) \in \C^2$. Assume also
that they are linearly independent at generic points (i.e. away from a proper analytic subset). If $Z$ is another
holomorphic vector field defined around $(0,0) \in \C^2$, there are uniquely defined {\it meromorphic}\, functions
$f$ and $g$ such $Z =fX + gY$. Indeed, setting $X = A \partial /\partial x + B \partial /\partial y$, $Y = C\partial /\partial x
+ D \partial /\partial y$ and $Z = P \partial /\partial x + Q \partial /\partial y$, we obtain:
\begin{equation}
f = \frac{PD - QC}{AD - BC} \; \; \, {\rm and} \; \; \, g = \frac{QA - PB}{AD - BC} \, . \label{veryelementary}
\end{equation}

The above formula allows us to describe all vector fields commuting with a given vector field $X$ provided
that we know one vector field $Y$ commuting with $X$ and not everywhere parallel to~$X$. In fact, for $Z = fX + gY$
we have
\begin{equation}
[X,Z]  =  [X,fX + gY] =  \frac{\partial f}{\partial X} X + \frac{\partial g}{\partial X} Y \, ,\label{forcommutation}
\end{equation}
since $[X,Y]=0$. Since $X$ and $Y$ are linearly independent at generic points, there follows
that $[X,Z]=0$ if and only if both $\partial f /\partial X$ and $\partial g /\partial X$ are identically zero.
In other words, a vector field $Z = fX + gY$
commutes with $X$ if and only if $f$ and $g$ are first integrals of $X$.

Having made the above remarks, consider
now a holomorphic vector field $X$ defined on an open set $U$ of some complex manifold. Recall that
$X$ is said to be semicomplete on $U$ if for every point $p \in U$ there exists a solution of $X$, $\phi : V_p \subset
\C \rightarrow U$, $\phi (0) =p$, $\phi' (T) = X (\phi (T))$, such that whenever $\{T_i \} \subset \C$ converges to a point
$\hat{T}$ in the boundary of $V_p$, the corresponding sequence $\phi (T_i)$ leaves every compact subset of $U$. In this way
$\phi : V_p \subset \C \rightarrow U$ is a maximal solution of $X$ in a sense similar to the notion of ``maximal solutions''
commonly used for real differential equations. A semicomplete vector field on $U$ gives rise to a semi-global flow $\Phi$
on $U$. A useful criterion to detect semicomplete vector fields can be stated as follows. First consider a holomorphic vector field
$X$ on $U$ and note that the local orbits of $X$ define a singular foliation $\fol$ on $U$. A regular leaf $L$ of $\fol$
is naturally a Riemann surface equipped with an Abelian $1$-form $dT_L$ which is called the {\it time-form}\, induced on $L$
by $X$. Indeed, at a point $p \in L$ where $X(p) \neq 0$, $dT_L$ is defined by setting $dT_L (p).X(p) =1$. Now,
according to \cite{JR1}, if $c: [0,1]\rightarrow L$ is an open (embedded) path then the integral
$$
\int_c dT_L
$$
is different from zero provided that $X$ is semicomplete.

As an application of the above criterion note that an explicit integration shows that
the vector field $z^k \partial /\partial z$ {\it is not} semicomplete on a neighborhood of $0 \in \C$ provided
that $k \geq 3$. More generally given a holomorphic vector field of the form $h(z) \partial /\partial z$, this
vector field can be written as $(z^k + {\rm h.o.t.}) \partial /\partial z$, where $k$ is the order of $h$ at
the origin. By using a ``perturbation argument'', the following statements can then be proved (cf. \cite{JR1} and \cite{JR}):

\begin{lema}
\label{one-dimensional_cases}
Consider a holomorphic vector field $X =h(z) \partial /\partial z$ defined around $0 \in \C$. Then the following holds:
\begin{itemize}
  \item If $h(0) =h'(0)=h''(0) =0$, then $X$ is not semicomplete around $0 \in \C$;

  \item If $X$ is semicomplete and $h(0) =h'(0)=0$, then the residue of $X$ around $0 \in \C$ is equal to zero.
\end{itemize}
\mbox{}\qed
\end{lema}

Next let $G$ represent $\C \times \C$ as a Lie group. The notion of semi-global flow fits in the setting
of ``maximal local actions'' in the sense of Palais, cf. \cite{Palais} and \cite{guillotFourier}. Consider
an open set $U$ of a complex manifold where a family $\lambda X + \mu Y$, $\lambda, \, \mu \in \C$ of pairwise
commuting vector fields is defined. With $G \simeq \C \times \C$,
we say that the family $\lambda X + \mu Y$ generates a maximal local action of $G$ if there exists a holomorphic map
$\Phi : \mathcal{V} \subset G \times U \rightarrow U$, $(0,0) \times U \subset \mathcal{V}$ satisfying the following conditions:
\begin{enumerate}
  \item $\Phi ((0,0) ,p) =p$ for every $p \in U$ and $\Phi ((t_1, s_1), \Phi ((t_2, s_2), p)) = \Phi ((t_1 +t_2, s_1 +s_2) ,p)$
  provided that both sides are defined.
  \item Given $p \in U$ and a sequence $\{ (t_i, s_i) \}$, with $((t_i, s_i),p) \subset \mathcal{V}$, verifying
  $\lim_{i \rightarrow \infty} (t_i,s_i) = (\hat{t}, \hat{s}) \in \partial \mathcal{V}$, the sequence $\Phi ((t_i, s_i),p)$ must
  leave every compact set contained in $U$.
\end{enumerate}
Note that the restriction of a maximal local action to every open subset of $U$ is still maximal on the set in question.
In particular we can talk about germs of locally maximal actions in the same way we talk about germs of semi-complete
vector fields.

The following lemma clarifies the nature of these definitions:

\begin{lema}
\label{generalequivalence}
Suppose that $X,Y$ are holomorphic vector fields satisfying $[X,Y] =0$. Then the following are equivalent:
\begin{enumerate}
  \item $X,Y$ are semi-complete.
  \item The family $\lambda X + \mu Y$, $\lambda, \, \mu \in \C$, defines a maximal local action.
  \item Every vector field $Z$ in the family $\lambda X + \mu Y$ is semicomplete.
\end{enumerate}
\end{lema}

\begin{proof}
Suppose that $X,Y$ are semi-complete and denote by $\Phi_X : \mathcal{V}_X \subset \C \times U \rightarrow U$
and by $\Phi_Y : \mathcal{V}_Y \subset \C \times U \rightarrow U$ their respective semi-global flows. Let us
define $\mathcal{U} \subset G \times U$ by saying that $(t,s,p)$ belongs to $\mathcal{U}$ if and only if $(t,p)$
belongs to $\mathcal{V}_X$ and $(s, \Phi_X (t,p))$ belongs to $\mathcal{V}_Y$. Next let $\mathcal{U}_0$ denote
the connected component of $\mathcal{U}$ containing $(0,0) \times U$. Finally let $\Phi_G :  \mathcal{U}_0 \subset
G \times U \rightarrow U$ be defined by $\Phi_G ((t,s),p) = \Phi_Y (s,\Phi_X (t,p))$. Note that $\Phi_G$ defines a
local $C^2$-action since $[X,Y]=0$ which is clearly generated by the family $\lambda X + \mu Y$.
In addition this local action is actually maximal since $X$ and $Y$ are semicomplete.
This shows that $(1)$ implies $(2)$
in our statement. To check that $(2)$ implies $(3)$, fix $Z = \lambda X + \mu Y$. A semi-complete flow $\Phi_{\lambda X}
: \mathcal{V}_{\lambda X} \subset \C \times U \rightarrow U$ (resp. $\Phi_{\mu Y} : \mathcal{V}_{\mu Y} \subset \C \times
U \rightarrow U$) for $\lambda X$ (resp. $\mu Y$) can therefore be obtained by suitable restriction of $\Phi_G$. Namely,
given $p \in U$, the map $\Phi_{\lambda X} (t,p)$ coincides with $t \mapsto \Phi_G ((\lambda t, 0) ,p)$. Analogously
$\Phi_{\mu Y} (s,p)$ coincides with $t \mapsto \Phi_G ((0, \mu s) ,p)$. The fact that $\Phi_{\lambda X} , \, \Phi_{\mu Y}$
are semi-global flows is an immediate consequence of the fact that $\Phi_G$ is maximal. Finally to produce a semi-global
flow $\Phi_Z$ associated to $Z$ it suffices to set $\Phi_Z (t,p) = \Phi_{\mu Y} (t, \Phi_{\lambda X} (t,p))$. Since it is
clear that condition $(3)$ implies condition $(1)$, the proof of our lemma is over.
\end{proof}

By virtue of the preceding lemma, we shall also say that the family $\lambda X + \mu Y$ is semicomplete to mean that it
generates a maximal local action of $G \simeq \C \times \C$.

As mentioned, the classification of semicomplete {\it holomorphic vector fields}\, around the origin of $\C^2$ was obtained
by Ghys and Rebelo in a series of papers (\cite{JR1}, \cite{JR}, and \cite{JR2}).
The paper \cite{GuR} casts these results in the far more general context of meromorphic vector fields, or rather
of birational theory of semicomplete vector fields. The results of \cite{GuR} however will not {\it strictly}\,
be needed in what follows.
In turn the results of \cite{JR1}, \cite{JR}, and \cite{JR2} are summarized by Table~\ref{TheTable}.

\begin{table}[htbp]
\centering
\begin{tabular}{p{0.05\textwidth}p{0.57\textwidth}p{0.38\textwidth}} \\ \toprule
& \multicolumn{2}{c}{Regular foliations} \\ \toprule
1 & $X = y^a F(x,y) \ddy$: &  $a \in \mathbb{N}$ \\
1.a &  $F(x,y) = 1$ & $a \ne 0$  \\
1.b &  $F(x,y) = x $& $a \ne 0$  \\
1.c & $F(x,y = x^2 + g_1(y) x + g_2(y)$ & $g_1$ and $g_2$ holomorphic on $(\mathbb{C}^2,0)$\\
& & $g_1(0) = g_2(0) = 0$\\
\toprule
& \multicolumn{2}{c}{Parabolic vector fields} \\ \toprule
2 & $X = f [x^2\ddx - y (nx-(n+1)y)\ddy]$ & $n\in \mathbb{N}$ \\
& & strictly meromorphic first integrals\\  \midrule
3 &  $X = f[(y-2x^2) \ddx - 2xy \ddy]$ & strictly meromorphic first integrals \\
& & first jet is nilpotent\\
\toprule
& \multicolumn{2}{c}{Elliptic vector fields} \\ \toprule
4 & $X = (xy(x-y))^a f [x(x-2y)\ddx + y(y-2x)\ddy]$ & $a\in \mathbb{N}$\\ \midrule
5 & $X = (xy(x-y)^2)^a f [x(x-3y)\ddx + y(y-3x)\ddy]$ & $a\in \mathbb{N}$\\ \midrule
6 & $X = (xy^2(x-y)^3)^a f [x(2x-5y)\ddx + y(y-4x)\ddy]$ & $a\in \mathbb{N}$\\ \midrule
7 & $X = (x^3+y^2)^a f [2y\ddx -3x^2\ddy]$ & $a\in \mathbb{N}$ \\ \midrule
8 & $X = (y(y-x^2))^a f [(2y-x^2)\ddx + 2xy\ddy]$ & $a\in \mathbb{N}$ \\ \midrule
9 & $X = (y(y-x^2))^a f [(3y-x^2)\ddx + 4xy\ddy]$ & $a\in \mathbb{N}$ \\
\toprule
& \multicolumn{2}{c}{Foliations with non-zero linear part} \\ \toprule
10 & $X= x^n y^m [(mx + {\rm h.o.t.}) \ddx - (ny + {\rm h.o.t.}) \ddy ]$ & $m, \, n \in \N^{\ast}$ \\ \midrule
11 & $X = x f [x\ddx + ny \ddy ]$& $n\in \mathbb{Z}^{\ast}$\\ \midrule
12 & $X = x^ay^b f[ m x\ddx - ny \ddy]$ & $n,m\in \mathbb{N}^\ast$\\
 & & $am-bn = \pm 1$\\
13 & $X = x^n y^n (x-y) f [x \ddx - y\ddy]$ & $n \in \mathbb{N}$\\
\bottomrule
\end{tabular}
\caption{Germs of semicomplete holomorphic vector fields on $(\mathbb{C}^2,0)$ with zero eigenvalues at the origin.}\label{TheTable}
\end{table}

\begin{obs}
\label{About_The_Table}
{\rm In Table~\ref{TheTable}, $f$ always stands for the germ of an invertible holomorphic function. Furthermore,
\begin{itemize}
\item In cases~$1$.a and~$1$.b we necessarily have $a \ne 0$ otherwise the resulting vector field $X$ would have non-zero
eigenvalues at $(0,0) \in \C^2$;

\item Parabolic vector fields only appear in the context of isolated singularities;

\item Elliptic vector fields have a non-constant holomorphic first integral given, in each case, by the ``polynomial raised
to power $a$'' in Table~\ref{TheTable}. Accordingly, the singular point is isolated if and only if $a=0$.

\item In both parabolic and elliptic cases, the corresponding
nilpotent vector fields can be obtained by collapsing $(-1)$-curves invariant
by the associated foliations as follows:
\begin{enumerate}
\item In the parabolic case the nilpotent vector field~$3$ is obtained out of the quadratic vector field~$2$ by setting $n=1$
and collapsing the resulting separatrix of self-intersection~$-1$;

\item In the elliptic case the vector fields~$5$ and~$6$ have separatrices of self-intersection~$-1$ as well. Once these
are collapsed, the vector fields~$8$ and~$9$ arise;

\item The vector field~$9$ still possesses a separatrix with self-intersection~$-1$ which can be collapsed to yield the vector
field~$7$.

\item As previously said, Proposition~\ref{non_invertible_function} sharpens Table~\ref{TheTable} in the sense that
it shows that in the case of parabolic vector fields as well as in a sub-case of item~11 the corresponding invertible
functions can always be made constant.
\end{enumerate}
\end{itemize}
}
\end{obs}

Let us close this section with a slightly technical lemma that will be useful for us to settle the
centralizer problem in the case of commuting vector fields $X$ and $Y$ such that the foliation associated with,
say, $X$ is regular.

\begin{lema}
\label{commuting_with_regularfoliation}
Assume that $Y$ is a (germ of) holomorphic semicomplete vector field whose eigenvalues at
$(0,0)$ are both equal to zero (so that $Y$ admits one of the normal
forms in Table~\ref{TheTable}). Assume also the existence of local coordinates $(u,v)$
where $Y$ takes on the form
$$
Y = A(u,v) \partial /\partial u + B(v) \partial /\partial v
$$
where $A$ and $B$ are holomorphic functions. Then either the foliation associated with
$Y$ is regular at the origin or $Y$ admits one of the following normal forms:
\begin{itemize}
  \item[(1)] $v h(v) [ nu \partial /\partial u + v \partial /\partial v]$ where $h (0) \neq 0$ and $n \in \Z^{\ast}$.

  \item[(2)] $h(v) [u (nv-(n+1)u) \partial /\partial u + v^2 \partial /\partial v]$ where $h (0) \neq 0$
  and $n \in \N$.
\end{itemize}
\end{lema}

\begin{proof}
Assume first that $Y$ has an isolated singularity at the origin. In this case, and in view of Table~\ref{TheTable},
it suffices to show that $Y$ can neither take on any of the elliptic normal forms (in Table~\ref{TheTable})
nor on the parabolic nilpotent vector field (item~3 in the mentioned table).
For this
we proceed as follows. Denote by $\fol$ the foliation associated with $Y$. Note that in the $(u,v)$-coordinates
the projection $(u,v) \mapsto v$ is transverse to $\fol$ away from the invariant axis $\{v=0\}$. In particular $\fol$
possesses a smooth separatrix ($\{v=0\}$) and this suffices to rule out the nilpotent elliptic vector field with a
cusp as separatrix, though we shall recover this fact from our more general argument.

Consider a leaf $L$ of $\fol$ different from $\{ v=0\}$ and a loop $c : [0,1] \rightarrow L$ whose projection on the
$u$-axis is denoted by $c_u$. Clearly the integral of the time-form $dT_L$ induced by $Y$ on $L$ over $c$ coincides
with the integral of the form $dv/B(v)$ over $c_u$. Now we have:

\noindent {\it Claim}. The integral of $dv/B(v)$ over $c_u$ equals zero.

\noindent {\it Proof of the Claim}. Note that $B(0) =B'(0) =0$ since both eigenvalues of $Y$ at the origin are
equal to zero. Moreover, we must have $B''(0) \neq 0$ since otherwise there is an open path $\overline{c} : [0,1]
\rightarrow \{ u=0\}$ over which the integral of $dv/B(v)$ equals zero (cf. Lemma~\ref{one-dimensional_cases}).
The last possibility however
cannot occur since a lift of $\overline{c}$ in a leaf of $\fol$ would provide us with an open path over which
the integral of the corresponding time-form equals zero, hence contradicting the assumption that $Y$ is semicomplete.

We can then assume that $B(v) = v^2 + {\rm h.o.t.}$ and the proof of the claim now follows
from the following observation: if this integral
is different from zero, then the endpoint of the loop can slightly be moved so as to produce a necessarily open
path over which the integral of $dv/B(v)$ is zero (cf. again Lemma~\ref{one-dimensional_cases}).
By lifting the resulting open path in a leaf of $\fol$ we again obtain
a contradiction with the fact that $Y$ is semicomplete. The claim is proved.\qed

Let us now show that $Y$ cannot take on any of the elliptic normal forms in Table~\ref{TheTable}.
This is however an immediate consequence of Lemma~\ref{nonzero_periods}
in the next section. Indeed, according to Lemma~\ref{nonzero_periods}, given an elliptic vector field
in Table~\ref{TheTable}, every neighborhood
$U$ of $(0,0) \in \C^2$ contains loops $c: [0,1]\rightarrow U$
lying in leaves of the foliation associated with $Y$ over which the integral of the corresponding time-form is
different from zero. This is clearly incompatible with the preceding claim.

To finish the proof of our lemma in the case of isolated singular points, it only remains to show that $Y$
cannot be conjugate to the nilpotent parabolic vector field in Table~\ref{TheTable}. Assume aiming at
a contradiction that this was the case. In particular, we would have
$Y = A(u,v) \partial /\partial u + B(v) \partial /\partial v$ with $A(u,v) = v
+{\rm h.o.t.}$ so that the blow-up $\widetilde{Y}$ of $Y$ is regular on the exceptional divisor. The foliation $\tilf$ associated
with $\widetilde{Y}$ has a unique singular point lying at the intersection of the exceptional divisor with
the transform of the axis $\{ v=0\}$. In coordinates $(u,t)$ such that $\pi (u,t) = (u,ut)$, this intersection point
is represented by the origin and $\widetilde{Y}$ is locally given by
$$
\widetilde{Y} = A (u,tu) \partial /\partial u + [\frac{-t}{u} A (u,tu) + \frac{1}{u} B(tu)] \partial /\partial t \, .
$$
This vector field has a quadratic singular point at $(0,0)$ and it must be conjugate to the quadratic parabolic
vector field with $n=1$, i.e. to the vector field $u (t-2u) \partial /\partial u + t^2 \partial /\partial t$
(item~2 in Table~\ref{TheTable}). In particular, the quadratic part of $\widetilde{Y}$ must
be linearly conjugate to $u (t-2u) \partial /\partial u + t^2 \partial /\partial t$.
However $B(tu)/u$ has order at least~$3$ so that the quadratic part of $\widetilde{Y}$ is actually given
by $A_2 (u,tu) \partial /\partial u -(t A_2 (u,tu)/u) \, \partial /\partial t$, where $A_2 (u,tu)$ stands for
the quadratic part of the function $(u,t) \mapsto A (u,tu)$. Therefore, in terms of foliation, the foliation associated
with the quadratic part of $\widetilde{Y}$ is simply $u \partial /\partial u - t \partial /\partial t$
and hence not (linearly) conjugate to the foliation associated with
$u (t-2u) \partial /\partial u + t^2 \partial /\partial t$. The resulting
contradiction completes the proof of the lemma in the case of isolated singular points.

Finally let us consider the case in which $Y$
possesses a curve of zeros through $(0,0) \in \C^2$. Since in $(u,v)$ coordinates every
function $g$ dividing $Y$ depends only on the variable~$v$, there follows that the zero-set of $Y$ consists of a single
smooth component through $(0,0)$. This rules out all the elliptic cases (multiplied by first integrals)
in Table~\ref{TheTable}. Thus the foliation associated with $Y$ is either regular at the origin or has integral
eigenvalues $m,n$ with $mn \neq 0$. In the latter case, a direct inspection in Table~\ref{TheTable} shows
that $Y$ must admit the form indicated in item~(1) above since, again, the zero-set of $Y$ is constituted by
a single smooth component. This completes the proof of the lemma.
\end{proof}

\section{Some global constructions}\label{globalexamples}

The purpose of this section is to solve the centralizer problem for the parabolic vector fields
and the elliptic vector fields appearing in Table~\ref{TheTable}.
The solution will be provided by means
of a geometric construction of the corresponding flows so that the realization problem mentioned in the
introduction will also be solved for the vector fields in question.
The constructions carried out here have an inevitable
overlap with the constructions conducted in \cite{JR} though the present version is slightly more accurate.

Let us begin with the case of elliptic vector fields where our purpose will be to prove the following
proposition:

\begin{prop}
\label{solving-centralizer_elliptic}
Let $X$ be an elliptic vector field as in Table~\ref{TheTable}. Then there is no holomorphic
vector field $Y$ forming a rank~$2$ system of commuting vector fields with $X$.
\end{prop}

The discussion conducted below will be slightly more general than what is strictly needed to prove
Proposition~\ref{solving-centralizer_elliptic}. In particular, we will show the existence of the mentioned
vector field $Y$ in the {\it meromorphic setting} i.e. if we are allowed to consider negative values of the exponent~``$a$''
in Table~\ref{TheTable}. We will also show that these meromorphic vector fields can be realized on compact surfaces
(where it should be emphasized that the realization on compact surfaces is as meromorphic vector fields).
Naturally the argument used at the beginning of Section~\ref{morelikebasics} remains valid for meromorphic
vector fields: once one vector field $Y$ forming a rank~$2$ commuting system with $X$ is known, all other
vector fields are obtained by means of the combinations
$$
f_1 X + f_2 Y
$$
where $f_1, f_2$ are meromorphic first integrals of~$X$.

To prove Proposition~\ref{solving-centralizer_elliptic} it suffices to deal with the case of quadratic
elliptic vector fields since the nilpotent elliptic vector fields can be obtained by blowing-down the
previous ones, cf. Remark~\ref{About_The_Table}. In turn, it suffices to consider the case of vector fields
admitting $xy(x-y)$ as first integral since the remaining cases can similarly be treated. Let us then begin
with
\begin{equation}
X = x(x-2y) \partial /\partial x + y(y-2x) \partial /\partial y \; . \label{The-vectorfield-X-elliptic}
\end{equation}
The foliation $\fol$ associated with~$X$ is given by the level curves of $xy(x-y)$. Thus $\fol$ can be
viewed in $\C P(2) \simeq \C^2 \cup \Delta$ (where $\Delta \simeq \C P(1)$ is the line at infinity)
as the pencil of elliptic curves given in the initial affine $\C^2$ by
$xy(x-y) = \alpha$, where $\alpha$ is a constant. The foliation $\fol$ has therefore $3$ singular points $p_1, p_2, p_3$ in the line
at infinity $\Delta$ which correspond to the intersections of $\Delta$ with the $x$-axis, the line $\{ x=y\}$,
and the $y$-axis, respectively. Apart from the union of the invariant lines $\{ y=0\}$, $\{ x=y\}$, and $\{ x=0\}$ all the
elliptic curves in the pencil in question pass through each of the singular points $p_i$ intersecting $\Delta$
with multiplicity~$3$, $i=1,2,3$.

To describe the structure of $\fol$ and of $X$ around the singular points $p_i$, $i=1,2,3$, we perform the
standard change of variables $(u,v) \rightarrow (1/u,v/u) = (x,y)$ so that $p_1$ coincides with the origin of
the $(u,v)$-coordinates. The first integral characterizing the leaves (elliptic curves) of $\fol$ then becomes
$v(v-1)/u^3$ while the vector field $X$ is now given by
$$
X_0 = \frac{1}{u} \left[ u(1-2v) \frac{\partial}{\partial u} + 3v (v-1) \frac{\partial}{\partial v} \right] \, .
$$
In particular, $X_0$ has poles of order~$1$ over $\Delta$. Since over a leaf of $\fol$ we have $v(v-1) = \alpha u^3$,
there follows easily that the restriction of $X_0$ to one of these leaves is actually regular (non-zero) at the origin: it
suffices to parameterize the leaf under the form $u=t$ and $v = \alpha t^3(1 + {\rm h.o.t.})$, $\alpha \neq 0$.
The analogous computations
also show that the restriction of $X_0$ to a leaf of $\fol$ is regular at the points $p_2$ and $p_3$ as well. In other words,
the restriction of $X$ to a non-degenerate elliptic curve in the pencil defined by $\fol$ is a nowhere zero holomorphic
vector field.

Up to blowing up each of the points $p_i$ three times, the pencil $\fol$ becomes an elliptic fibration $\calp$ on a surface
$M$ fibering over $\C P(1)$. This fibration has exactly two singular fibers, namely the fiber $\calp^{-1} (0)$
over $0$ (given by the union of the invariant lines $\{ y=0\}$, $\{ x=y\}$, and $\{ x=0\}$)
and the fiber $\calp^{-1} (\infty)$ over $\infty$ which is the singular
fiber~$IV^{\ast}$ in Kodaira's table (see for example \cite{kodaira}, \cite{GuR}). Furthermore the vector field $X_0$ (identified
with its own transform) has poles over $\calp^{-1} (\infty)$ and is holomorphic on $M \setminus \calp^{-1} (\infty)$.
As previously seen, $X_0$ is tangent to the fibers and restricted to a regular fiber of $\calp$ is holomorphic and nowhere zero:
since the fiber is an elliptic curve, there follows that $X_0$ is constant over this fiber and, in particular, complete.
Summarizing the preceding, we have:
\begin{itemize}
  \item The foliation $\fol$ associated with $X_0$ yields an elliptic fibration $\calp$ on a compact surface $M$ fibering
  over $\C P(1)$ with exactly two singular fibers, $\calp^{-1} (0)$ and $\calp^{-1} (\infty)$ (of types respectively~$IV$
  and~$IV^{\ast}$  in Kodaira's table, \cite{kodaira}, \cite{GuR}).

  \item $X_0$ is holomorphic and complete on the open surface $N = M \setminus \calp^{-1} (\infty)$.

  \item If $X_0$ is multiplied by a first integral, i.e. by a meromorphic function on the basis
  $\C P(1)$ having poles only at $\{ \infty \}$, then $X$ still is a complete vector field on the open surface~$N$.

  \item all the regular elliptic fibers of $\calp$ are isomorphic as elliptic curves. Indeed, the holomorphic map
  from $\C^{\ast}$ to the moduli space of elliptic curves that assigns to a point $z \in \C^{\ast}$ the point
  in the moduli space determined by the elliptic curve $\calp^{-1} (z)$ must be constant since the moduli
  space in question is complex hyperbolic.
\end{itemize}

\begin{lema}
\label{nonzero_periods}
Consider $X$ defined on $\C^2$ and its associated foliation $\fol$. For every regular leaf $L$ of
$\fol$ different from the three invariant lines, $X$ has non-zero periods on $L$. Furthermore the periods
of $X$ vary from leaf to leaf so as to give rise to a non-constant holomorphic function on the corresponding
leaf space.
\end{lema}

\begin{proof}
The first assertion is clear. In fact, if we place ourselves in the context of the elliptic surface $M$,
it was seen that the restriction of $X$ to a regular fiber in $M$ is a (non-zero) constant vector field
on an elliptic curve. The periods are therefore non-zero.

From the description of $X$ as a vector field tangent to the fibers of $\calp$, it is clear that the periods
of $X$ provide holomorphic functions on the corresponding leaf space which can naturally be identified with
$\C^{\ast}$. It only remains to show that these functions are not constant.

Consider then $X$ and $\fol$ in $\C^2$. Fix a leaf $L$ of $\fol$ and let $c \subset L$ be a loop over which the
integral of the time-form associated with $X$ is different from zero. Choose $\lambda \in \C^{\ast}$ and consider
the homothetic map $\Lambda : (x,y) \mapsto (\lambda x ,\lambda y)$. If $\lambda$ is small enough, then
$\Lambda (c)$ is contained in arbitrarily small neighborhoods of $(0,0) \in \C^2$. Moreover $\Lambda$ preserves
the foliation $\fol$ so that $\Lambda (c)$ is still a loop contained in a certain leaf $L_{\lambda}$ of
$\fol$. If we denote by $dT_L$ (resp. $dT_{L_{\lambda}}$) the time-form induced by $X$ on $L$
(resp. $L_{\lambda}$), we clearly have
$$
\int_c \Lambda^{\ast} (dT_{L_{\lambda}}) = \int_{\Lambda (c)} dT_{L_{\lambda}} \; .
$$
However $\Lambda^{\ast} X = \lambda X$ since $X$ is homogeneous of degree~$2$. Thus
$\Lambda^{\ast} (dT_{L_{\lambda}}) = \lambda^{-1} dT_L$. In other words,
$$
\int_{\Lambda (c)} dT_{L_{\lambda}} = \frac{1}{\lambda} \int_c dT_L \, .
$$
Thus the period of $X$ over $\Lambda (c)$ becomes unbounded as $\lambda \rightarrow 0$. Since over two different (elliptic)
leaves of $\fol$ the corresponding restrictions of $X$ differ only by a multiplicative constant (the leaves in question
being pairwise isomorphic), the lemma follows.
\end{proof}

\begin{obs}
\label{weneverknow}
{\rm Note that the multiplicative constant arising from comparing $X_0$ restricted to two different
elliptic fibers is further affected if $X_0$ is multiplied by a first integral. Since there are holomorphic
first integrals that do not vanish at $(0,0)$, we recover the fact that the invertible multiplicative function
appearing ``in front'' of these vector fields in Table~\ref{TheTable} cannot be made constant in general.

Moreover, if the homology class of a loop $c$ in the elliptic fibers is fixed (and recalling that they are all isomorphic
as elliptic curves), the multiplicative constant relating the restrictions of $X_0$ to two different fibers
can easily be obtained by comparing the corresponding periods over~$c$. This clearly gives rise to a holomorphic
first integral~$\mathcal{I}$ for $X_0$ (or for $\fol$) defined on $\C^{\ast}$ identified with the base space of the regular part
of the fibration $\calp$. The argument used in the end of the proof of Lemma~\ref{nonzero_periods} shows
that $\mathcal{I}$ has a pole at $0 \in \C$ while it is holomorphic (equal to zero) at infinity.}
\end{obs}

\begin{proof}[Proof of Proposition~\ref{solving-centralizer_elliptic}]
Let then $X_0$ be as above and consider a vector field $X$ obtained by multiplying $X_0$ by an invertible function
$f$ and by a suitable (non-negative) power of $xy(x-y)$. As previously seen, on arbitrarily small neighborhoods
of $(0,0) \in \C^2$, we can find loops $c$ inside leaves $L$ of $\fol$ such that the integrals of $dT_L$ over $c$
become unbounded - in particular the period of $X$ does vary with the leaf (where $dT_L$ stands for the corresponding time-forms).
If there were a vector field $Y$ forming a commutative rank~$2$ system with $X$, the (local) flow of $Y$ would permute
these leaves of $\fol$. Moreover, this local flow preserves $X$ which means that the periods of $X$ would have to be
independent of the leaf chosen. The resulting contradiction proves Proposition~\ref{solving-centralizer_elliptic}
in the case of holomorphic vector fields whose underlying foliation has $xy(x-y)$ as first integral. The proof
of this proposition in the other two (quadratic) cases is however totally analogous and hence left to the reader.
Finally, as already pointed out, the nilpotent cases follow from the quadratic ones.
The proof of Proposition~\ref{solving-centralizer_elliptic} is complete.
\end{proof}

Note that if the vector field $X_0$ is allowed to be multiplied by the above considered function $\mathcal{I}$, then we
obtain vector fields whose periods no longer vary with the fibers. More precisely, $\mathcal{I}$ lifts to a first integral
$I$ of $X_0$ since $\mathcal{I}$ is a function defined on the leaf space of $\fol$. Clearly $I$ must be meromorphic around
the origin since $\mathcal{I}$ has a pole at $0 \in \C$. Up to multiplying $X_0$ by $I$, we obtain a meromorphic vector field
$X$ with poles exactly over the $3$ invariant lines through $(0,0)$. The restrictions of $X$ to the leaves of $\fol$ however
all have the same period. Indeed, the lift of $X$ to the surface $M$ always gives the same vector field on each of the elliptic
fibers.

At this point, we can wonder again about the existence of a (possibly meromorphic) vector field $Y$ forming
a rank~$2$ commutative system with~$X$.
To construct a vector field $Y$ as desired, we first go back to the compact elliptic surface $M$.
The monodromy group of $M$ is cyclic generated by a loop around one of the singular
fibers and the resulting monodromy map is well known (see for example \cite{barthpeters},
page 210). It turns out that the fibers are isomorphic to the elliptic curve $\mathcal{E}$
obtained as the quotient of $\C$ by the group generated by~$1$ and by~$\exp(2i\pi/3)$. An explicit model
for $M$ - or at least for $N = M \setminus \calp^{-1} (\infty )$ - can be obtained as follows.
Consider the product of $\C \times \mathcal{E}$ of $\C$ and the elliptic curve $\mathcal{E}$.
Note that $\mathcal{E}$ viewed as a quotient of $\C$ is stable by multiplication by~$\exp(2i\pi/3)$.
Indeed this multiplication induces an automorphism of $\mathcal{E}$ of order~$3$ having~$3$ fixed
points. Next let $\sigma$ be the diffeomorphism of $\C \times \mathcal{E}$ given by
$(x,y) \mapsto (\exp(2i\pi/3) x , \exp(2i\pi/3) y)$. This diffeomorphism has again order~$3$ and all of its
fixed points lie in the curve $\{ 0 \} \times \mathcal{E}$. Thus the quotient of $\mathbb{D} \times \mathcal{E}$
by the diffeomorphism $\sigma$ is a singular elliptic surface: the singular points are in number
of~$3$ and correspond to the fixed points in $\mathcal{E} \simeq \{ 0 \} \times \mathcal{E}$ of the automorphism
induced by multiplication by~$\exp(2i\pi/3)$. The singular points however can easily be resolved: it suffices
to perform a single blow-up which leads to a smooth surface containing~$4$ rational curves, namely: the three
exceptional divisors (rational curves of self-intersection~$-3$) and a rational curve of self-intersection~$-1$
corresponding to the quotient of $\mathcal{E} \simeq \{ 0 \} \times \mathcal{E}$. Once the latter rational
curve is collapsed, the resulting surface is nothing but the open surface~$N$.

The above construction makes it clear that $N \setminus \calp^{-1} (0)$ is endowed with a natural holomorphic
connection $\nabla$. In fact, the horizontal connection on $\C \times \mathcal{E}$ is preserved by the action
of $(x,y) \mapsto (\exp(2i\pi/3) x , \exp(2i\pi/3) y)$ and hence induces a connection $\nabla$ on the quotient.
To obtain the desired vector field $Y$ we now proceed as follows. Consider on $\C$ identified to the basis of
the elliptic fibration $\calp$ on $N$ the vector field $x \partial /\partial x$. Now using $\nabla$
lift $x \partial /\partial x$ to a holomorphic vector field $Y$ on $N \setminus \calp^{-1} (0)$ and consider its
extension to all of $N$. To show that $[X,Y]=0$ it suffices to check that the flow of $Y$ preserves the
vector field $X$. It is however clear that the flow of $Y$ takes fibers of $\calp$ to fibers of $\calp$
by construction. Thus it pulls-back the restriction of $X$ to a certain fiber to another fiber. The pulled-back
vector field however has the same period than the original vector field on the fiber in question. They must
therefore coincide and this proves that $[X,Y]=0$.

The remainder of this section will be devoted to the centralizer and to the realization problems in the case of
the parabolic vector fields in Table~\ref{TheTable}. Note that the automorphism group of Hirzebruch surfaces
is described for example in \cite{Akhiezer} and again the material discussed below has some intersection
with \cite{JR}.

Let $n \in \N$ be fixed and consider two copies of $\C \times (\C \cup \{ \infty \})$ with coordinates
$(x,y)$ and $(u,v)$ where $y, v \in \C \cup \{ \infty \}$. For $x\neq 0$, let the point
$(x,y)$ of the first copy to be identified with the point $(1/x , y/x^n) = (u,v)$ of the second one. The result of this
gluing in the Hirzebruch surface $F_n$ which can also be described as the (fiberwise) compactification of the
line bundle with Chern class~$-n$ over $\C P(1)$. In particular, $F_0$ is isomorphic to the product $\C P(1) \times
\C P(1)$. Similarly, the surface $F_1$ is not minimal in the sense that it contains a $(-1)$-rational curve: by
collapsing this curve we obtain the projective plane $\C P(2)$. Furthermore, for every $n \in \N$, the parabolic
vector field $X_{n}$ in Table~\ref{TheTable} can be globalized in the surface $F_n$ (cf. \cite{Akhiezer}, \cite{JR}).

Once more it suffices to work with the (quadratic) parabolic vector fields $X_n$ since the nilpotent parabolic
vector field $P$ can be obtained out of $X_1$ by collapsing the $(-1)$-curve in $F_1$. Our purpose will be to
characterize vector fields $Y$ forming a rank~$2$ commutative system with $X_n$ as well as to show that the
corresponding germs of $\C^2$-actions are all realized by globally defined vector fields on $F_n$.

Consider the flow $\Phi^t$ defined on the first copy $\C \times (\C \cup \{ \infty \})$
by $\Phi^t (x,y) = (x+t , y + (x+t)^{n+1} - x^{n+1} )$. Let also $\Psi^s$ be the flow on $\C \times (\C \cup \{ \infty \})$
given by $\Psi^s (x,y) = (x, y+s)$. It is immediate to check that these two flows commute, i.e.
$\Psi^s \circ \Phi^t (x,y) = \Phi^t \circ \Psi^s (x,y)$ for every $t,s \in \C$.

Now it is straightforward to check that in $(u,v)$-coordinates, $\Phi^t$ takes on the form
$$
\Phi^t (u,v) = \left( \frac{u}{1+tu} , (1 +tu)^{-n} \left[ v + \frac{1}{u} [(1+tu)^{n+1} -1] \right] \right) \, .
$$
Similarly, $\Psi^s (u,v) = (u, v+su^n)$. In particular both $\Phi$ and $\Psi$ are global holomorphic flows
on $F_n$. Moreover the above expressions show that the point $p = \{ u=0 , v = \infty \}$ is fixed by both flows so that
the holomorphic vector fields $Z$ and $Y$ arising respectively from $\Phi$ and $\Psi$ have
a singular point at~$p$. Finally the vector fields $Z$ and $Y$ clearly verify $[Z,Y]=0$ since the flows $\Phi$ and $\Psi$
commute. Next note that
in coordinates $\overline{u} = u$ and $\overline{v} =1/v$, the point $p$ is identified with the origin and
a direct verification shows that the vector field $Z$ becomes
$$
Z = \overline{u}^2 \partial / \partial \overline{u} - \overline{v} (n \overline{u} - (n+1) \overline{v})
\partial /\partial \overline{v} \, .
$$
Hence the germ of $Z$ at~$p$ coincides with the vector field $X_n$. In other words, $Z$ is a global
realization of the vector field $X_n$ on the compact surface $F_n$. Furthermore, the vector field $Y$
is given in the coordinates $(\overline{u} , \overline{v})$ by
$$
Y = - \overline{u}^n \overline{v}^2 \partial /\partial \overline{v} \, .
$$
Now recalling Formula~(\ref{forcommutation}) and taking into account that the only (non-constant)
first integral of the vector field $X_n$ is strictly meromorphic, what precedes can be summarized as follows:

\begin{prop}
\label{solving-centralizer_parabolic}
Let $X_n$ be a parabolic (quadratic) vector field as in Table~\ref{TheTable}. Then every vector field
$Y$ forming a rank~$2$ system of commuting vector fields with $X$ has the form
$$
Y = c_1 \overline{u}^n \overline{v}^2 \partial /\partial \overline{v} + c_2 X_n
$$
with $c_1, c_2 \in \C$. Furthermore, all the corresponding germs of $\C^2$-actions can be realized
by global vector fields on the compact surface $F_n$.\qed
\end{prop}

\begin{obs}
\label{thenilpotent_vectorfield_P}
{\rm As previously mentioned, the case of the parabolic nilpotent vector field $P$ can be derived from
Proposition~\ref{solving-centralizer_parabolic} when $n=1$. Indeed, every holomorphic vector field
forming a rank~$2$ system of commuting vector fields with $P$ has the form
$$
c_1 y(x \partial /\partial x + y \partial /\partial y) + c_2 P
$$
with $c_1, c_2 \in \C$. Furthermore, all the corresponding germs of $\C^2$-actions can be realized
by global vector fields on $\C P(2)$.}
\end{obs}

\section{On the invertible functions in Table~\ref{TheTable} and the centralizer problem}\label{invertiblefunctions_etc}

In this section we will further advance the proof Main Theorem by essentially settling the cases
of vector fields whose underlying foliation {\it is not}\, of Martinet-Ramis type.

We begin by considering vector fields in Table~\ref{TheTable} which admit a {\it strictly meromorphic}\, first
integral, namely the parabolic vector fields in items~2 and~3 as well as vector fields of the form
\begin{equation}
x f [x\ddx + ny \ddy ] \label{dicritical_sing}
\end{equation}
where $n \in \N^{\ast}$; cf. item~11 in Table~\ref{TheTable}. Proposition~\ref{non_invertible_function} below
consists of a slight improvement on the content of Table~\ref{TheTable} itself.

\begin{prop}
\label{non_invertible_function}
Let $X$ be a holomorphic semicomplete vector field defined on a neighborhood of $(0,0) \in \C^2$. Assume
that both eigenvalues of $X$ at the origin are equal to zero. Assume also that $X$ possesses a strictly meromorphic
first integral. Then $X$ possesses one of the following normal forms:
\begin{enumerate}
  \item $x [x\ddx + ny \ddy ]$ with $n \in \N^{\ast}$;
  \item $x^2\ddx - y (nx-(n+1)y)\ddy$ with $n\in \mathbb{N}$;
  \item $(y-2x^2) \ddx - 2xy \ddy$.
\end{enumerate}
In other words, the invertible multiplicative function appearing in front of the models~2, 3, and~11 - this last
case only if $n >0$ - in Table~\ref{TheTable} can be made constant.
\end{prop}

\begin{proof}
Let $X$ be a holomorphic semicomplete vector field as in the statement. We can assume once and for all that $X$
has the form $X =fZ$ where $f$ is an invertible function and where $Z$ is a vector field belonging to the above
indicated list ((1), (2), and (3)) of vector fields.
In particular, both $X$ and $Z$ share the same associated foliation which will be denoted by $\fol$.
The proof of the proposition consists of finding coordinates where the function $f$ becomes constant.

Recall that a {\it separatrix}\, for $\fol$ is an irreducible analytic curve $\mathcal{C}$ passing through
the origin which is invariant under $\fol$. In other words, $\mathcal{C} \setminus \{ (0,0) \}$ is a leaf
of $\fol$. Next note that all the leaves of $\fol$ actually define separatrices for this foliation as it easily
follows from the fact that $\fol$ possesses a {\it strictly meromorphic}\, first integral.

Let then $L$ be a leaf of $\fol$ and consider the restriction of $X$ to $L$. In a local coordinate $z$ along $L$
obtained from a suitable Puiseux parameterization, the restriction of $X$ to $L$ has the form $(z^2 + {\rm h.o.t.})
\partial /\partial z$. Since this restriction must be semicomplete, there follows from Lemma~\ref{one-dimensional_cases}
that the residue of $X$ around $z=0$ is zero and hence that this vector field is conjugate to $z^2 \partial /\partial z$.
Since the same applies to the restriction of $Z$ to $L$, we conclude that the restrictions of $X$ and $Z$ to
a same leaf $L$ of $\fol$ are always conjugate. Thus to prove the proposition, we only need to check that it is possible
to coherently patch together these ``foliated'' change of coordinates so as to produce an actual change of coordinates
on a neighborhood of $(0,0) \in \C^2$.

To construct the desired change of coordinate out of the fact that the restrictions of $X$ and $Z$
to a same leaf $L$ of $\fol$ are conjugate, we proceed as follows. Consider first the case where $Z$ is as in item~(1)
of the list in the statement, with $n =1$.
In other words, $Z = x (x \ddx + y \ddy)$ and $\fol$ is given by the vector field
$x \ddx + y \ddy$ whose leaves are radial lines through the origin. Denote by $\tilf$ the blow-up of $\fol$
at the origin and let $\pi^{-1} (0) \simeq \C P(1)$ be the exceptional divisor. The leaves of $\tilf$
are transverse to $\pi^{-1} (0)$ so that the space of these leaves is naturally identified to $\C P(1)$.
There is however no holomorphic section defined on the leaf space and taking values in a neighborhood of the exceptional
divisor since the self-intersection of $\pi^{-1} (0)$ is strictly negative
(equal to~$-1$ in this case).
On the other hand, the fact that the self-intersection is strictly negative
allows us to apply the holomorphic tubular neighborhood theorem of Grauert \cite{grauert} and thus we can identify
a neighborhood of $\pi^{-1} (0)$ with a neighborhood $U$ of the null section in the corresponding normal line bundle.

Note that if there were a section $\sigma_1$ of the mentioned line bundle (with values on $U$), then
the desired change of coordinates would be obtained as follows: for every $p \in \pi^{-1} (0) \simeq \C P(1)$,
consider the point $\sigma_1 (p) \in L$ and construct a holomorphic diffeomorphism between the restrictions
of $X$ and of $Z$ to $L$ by means of the semi-global flows $\Phi_X$ and $\Phi_Z$
of these ($1$-dimensional) vector fields.
More precisely, since both restrictions are conjugate to the $1$-dimensional vector field $z^2 \partial /\partial z$,
for every $q \in L \setminus \pi^{-1} (0)$, there is a unique $T_{q,X} \in \C$ (resp. $T_{q,Z} \in \C$)
such that $\Phi_X^{T_{q,X}} (\sigma_1 (p)) =q$ (resp. $\Phi_Z^{T_{q,Z}} (\sigma_1 (p)) =q$). A diffeomorphism
$h_L$ of $L$ conjugating the
restrictions of $X$ and of $Z$ to $L$ can then be obtained by setting $q \mapsto \Phi_Z^{T_{q,X}} \circ
\Phi_X^{-T_{q,X}} (q)$. Since $\sigma_1$ is holomorphic, these ``foliated'' diffeomorphisms glue together on
a diffeomorphism defined on $U \setminus \pi^{-1} (0)$ which would then extend to all of $U$ owing to Riemann's theorem.

Whereas no holomorphic section $\sigma_1$ as above exists, the preceding idea can be adapted by using
a suitable sequence of meromorphic sections. We begin by fixing a point $p_N \in \pi^{-1} (0)$.
The leaf of $\tilf$ through $p_N$ will be denoted by $L_{p_N}$. Next let
$V_1$ be a relatively compact disc contained in $\pi^{-1} (0) \setminus \{ p_N \}$. On the open set $V_1$
we can find a holomorphic section $\sigma_{1}$ as above. Indeed, $\sigma_1$ can be the restriction of a meromorphic section
of the normal bundle having poles only at $p_N$. By using $\sigma_1$ as indicated in the preceding paragraph, we
then construct a diffeomorphism $h_1$ defined on the subset of $U$ lying over $V_1$. To be more
precise, for every $q$ in the open
set in question, the restriction of $h_1$ to the leaf of $\tilf$ through~$q$ is given by
$q \mapsto \Phi_Z^{T_{q}} \circ \Phi_X^{-T_{q}} (q)$, where $T_q$ is the unique complex time for which
$\Phi_X^{T_q} (\sigma_1 (p)) = q$.

It is however clear that the image of the (globally defined meromorphic) section $\sigma_1$
will eventually leave the neighborhood $U$ of $\pi^{-1} (0)$ fixed in the beginning if the open set $V_1$ is
sufficiently enlarged. Thus we cannot immediately
guarantee that $h_1$ is defined on $U \setminus L_{p_N}$. To overcome this difficulty, consider then another
relatively compact disc $V_2 \subset \pi^{-1} (0) \setminus \{ p_N \}$ containing $V_1$. Assume also that the
graph of $\sigma_1$ over $V_2$ is not contained in $U$, otherwise there is nothing to be proved.

Now $\sigma_1$ can be deformed into another section $\sigma_2$ by using the flow of $X$, i.e. by considering
sections of the form $\Phi_X^T \circ \sigma_1$ which are still (global) meromorphic sections with poles only at
$p_N$. Here a comment is needed since the image of $\sigma_1$ may not be contained in the domain of definition of~$X$.
The simplest way to overcome this difficulty consists of defining sections over $V_1$
under the form $\Phi_X^T \circ \sigma_1$ and then observing that these sections can naturally be extended to global
meromorphic sections as desired. To check that these ``local'' sections can be extended into meromorphic ones,
note that  the leaves of $\tilf$ can be identified with open discs in the fibers of the normal bundle in question.
The vector field $z^2 \partial /\partial z$ is however globally defined in the corresponding fibers and hence
can be used to move the section $\sigma_1$. Finally since this vector field is conjugate to $X$ over $V_1$,
there is no essential difference between deforming $\sigma_1$ by $\Phi_X^T \circ \sigma_1$ or by the flow
of $z^2 \partial /\partial z$.

In view of the preceding, by a small abuse of notation, we will keep talking about global meromorphic sections of the form
$\Phi_X^T \circ \sigma_1$ in what follows. As mentioned, these sections are holomorphic away from $p_N$.
Furthermore,
if $T$ is suitably chosen to ``bring the graph of $\sigma_1$ closer to the null section'', then the graph
of $\sigma_2$ over the disc $V_2$ will still be contained in $U$. Let then $T_{12}$ stand for the
complex time chosen in the definition of $\sigma_2$, i.e. $T_{12}$ is such that $\sigma_2 = \Phi_X^{T_{12}} \circ \sigma_1$.
In particular, for every point $p \in V_1$, we can consider the points $\sigma_2 (p) =
\Phi_X^{T_{12}} \circ \sigma_1 (p)$ and $p^{\ast} = \Phi_Z^{T_{12}} \circ \sigma_1 (p)$. These two points
are related by $h_1$, more precisely we have $h_1 (\sigma_2 (p)) = p^{\ast}$. However we will need to consider the points
$\sigma_2 (p)$ and $p^{\ast}$ for every $p \in V_2$ (and not only in $V_1$). The difficulty here is analogous to
the difficulty in defining meromorphic sections of the form $\Phi_X^T \circ \sigma_1$ already mentioned above. However, considering
that $\sigma_2$ is already defined, the point $p^{\ast}$ can simply be taken as
$p^{\ast} = \Phi_Z^{T_{12}} \circ \Phi_X^{-T_{12}} \circ \sigma_2 (p)$.

By using $\sigma_2$ and the above construction, the next step is to
define an extension $h_2$ of $h_1$ to the subset of $U$ lying over
$V_2$. For this, consider a leaf of $\tilf$ through a point $p$ (identified with the leaf space)
in $V_2$ and consider a point $q$ in this leaf.
Denoting by $T_{q}^{(2)}$ the time at which the flow of $X$ takes $\sigma_2 (p)$ to $q$, a diffeomorphism
conjugating the restrictions of $X$ and $Z$ to the leaf in question is obtained by setting
$q \mapsto \Phi_Z^{T_{q}^{(2)}} (p^{\ast})$. This collection of diffeomorphisms defined on individual leaves gives rise
to a diffeomorphism $h_2$ defined on the subset of $U$ lying over $V_2$. Finally, to check that
$h_2$ is an extension of $h_1$ consider $p \in V_1$. With the above notation, the time $T_q^{(2)}$
satisfies then $T_q^{(2)} = T_q - T_{12}$. Hence
$$
\Phi_Z^{T_{q}^{(2)}} (p^{\ast}) = \Phi_Z^{T_{q}} \circ \Phi_Z^{-T_{12}} (p^{\ast}) = \Phi_Z^{T_{q}} \circ \sigma_1 (p)= h_1 (q) \, .
$$
Therefore $h_2$ coincides with $h_1$ over $V_1$ as desired.

By considering an exhaustion of $\pi^{-1} (0) \setminus \{ p_N \}$ by open discs $V_k$ as above, we can
define a diffeomorphism $h$ conjugating $X$ and $Z$ in $U \setminus L_{p_N}$. This diffeomorphism however
is clearly bounded on neighborhoods of points in $L_{p_N} \setminus \{ p_N \}$ since it is preserves the leaves
of the common foliation $\fol$ and, in restriction to these leaves, sends $X$ to $Z$. Thus Riemann's theorem
implies that $h$ has a holomorphic extension to $L_{p_N} \setminus \{ p_N \}$. Finally it also extends
to $p_N$ since $p_N$ is an isolated point in the $2$-dimensional ambient space. This completes
the proof of the proposition in the case where $\fol$ is given by the vector field $x \ddx + y \ddy$.

The preceding proof can however be extended to the remaining cases. Assume first that the foliation
$\fol$ is given by $x\ddx + ny \ddy$ ($n \geq 2$). After performing finitely many blow-ups, we arrive
to a radial singularity of the form $x \ddx + y \ddy$ and we can then construct the
desired diffeomorphism on a neighborhood of the singular point in question. This diffeomorphism will naturally
be defined on the complement of finitely many leaves (separatrices) for the blown-up foliation. In fact,
apart from the radial singularity in question, all singularities of the blown-up foliation are linear
and of Siegel type. Hence the saturated of the radial singularity fills all of $U$ bar finitely many leaves
arising as separatrices for the Siegel singularities in question. As above, these diffeomorphisms are bounded
on neighborhoods of points in these separatrices that are not contained in the exceptional divisor. Thus Riemann's
theorem allows us to extend the diffeomorphism to each punctured separatrix and then to the whole space.
This shows the existence of a diffeomorphism
conjugating $X$ and $Z$ for foliations of the form $x\ddx + ny \ddy$ ($n \geq 2$).

Consider now the case where the foliation $\fol$ is given by $x^2\ddx - y (nx-(n+1)y)\ddy$ ($n\in \mathbb{N}$).
Denote by $\tilf$ the blow-up of $\fol$ at the origin and note that $\tilf$ leaves invariant the exceptional
divisor $\pi^{-1} (0)$. Moreover $\tilf$ possesses exactly $3$ singular points, namely:
\begin{itemize}
  \item A singularity $p_1$ where $\tilf$ admits $x\ddx + (n+1) y \ddy$ as local model.

  \item Two linearizable singularities $p_2$ and $p_3$ of Siegel type. Furthermore the eigenvalues of $\tilf$
  at $p_2$ are $1$ and $-1$ whereas the eigenvalues of $\tilf$ at $p_3$ are $1$ and $-(n+1)$.
\end{itemize}

The previous results allow us to construct the diffeomorphism $h$ conjugating $X$ and $Y$ on a neighborhood
of the singular point $p_1$. However the saturated of this neighborhood by $\tilf$ covers a neighborhood of
the exceptional divisor bar the separatrices transverse to $\pi^{-1} (0)$ of $\tilf$ at $p_2$ and $p_3$. As previously
indicated, Riemann's theorem can then be used to extend $h$ to the two separatrices in question. Finally,
the remaining case of the nilpotent vector field (item~(3) in the statement) follows immediately since it can be
obtained as the blow-down of the vector field $x^2\ddx - y ( nx-(n+1)y)\ddy$ with $n=1$. The proof
of Proposition~\ref{non_invertible_function} is finished.
\end{proof}

At this point the progress made so far in the centralizer problem can be summarized by reducing the proof
of Main Theorem to the proof of Theorem~\ref{harder_cases} below:

\begin{teo}
\label{harder_cases}
Assume that $X, \, Y$ form a rank~$2$ system of commuting semicomplete holomorphic
vector fields on $(\C^2, 0)$. Assume also that the eigenvalues of both $X$ and $Y$ at the origin are zero.
Then we have:
\begin{itemize}
  \item[(A)] If the foliation associated with one of these vector fields is regular. Then, up to linear combination,
  there are coordinates where the pair $X$ and $Y$ takes on one of the following forms:
  \begin{itemize}
    \item[$\bullet$] $X = y^n \ddx$ and $Y = y (a(x,y) \ddx + b(y) \ddy)$ where $b(y) = \alpha y + {\rm h.o.t.}$ with
    $\alpha \in \C^{\ast}$ and $a(x,y) =1 + \alpha n x +xr(y) +s(y)$ where $r, s$ are holomorphic functions
    satisfying $r(0) =s(0) =0$;

    \item[$\bullet$] $X = x^2 \ddx$ and $Y = y^2 \ddy$;

    \item[$\bullet$] $X = y^n \ddx$ and $Y = y(nx \ddx +y\ddy)$, with $n \in \N^{\ast}$;

    \item[$\bullet$]
    $$
    X = g_1 (xy^n) y^nx^2 \left[ \frac{\partial}{\partial x} + y^{n+1} g_2 (xy^n)  \left( nx \frac{\partial}{\partial x}
    - y \frac{\partial}{\partial y} \right) \right]
    $$
    and $Y = y(nx \ddx +y\ddy)$, where $n \in \N^{\ast}$ and $g_1$, $g_2$ are holomorphic functions of a single variable
    with $g_1 (0) =1$ and $g_1'(0) = g_2 (0) =0$;

    \item[$\bullet$] $X = y^n x^2 \ddx$ and $Y = x (ny-(n+1)x)\ddx - y^2\ddy$, with $n\in \mathbb{N}$.
  \end{itemize}

  \item[(B)] If the foliations associated with both $X$ and $Y$ are of Martinet-Ramis type. Then, up to linear combination,
  there are coordinates where the pair $X$ and $Y$ takes on the following form:
  $$
  X = (x^n y^m)^{k_1} x^{a_{\mu}} y^{b_{\mu}} [mx \partial /\partial x - ny  \partial /\partial y] \, ,
  $$
  $k_1 \in \N$, where $a_{\mu}, b_{\mu}$ are non-negative integers satisfying $a_{\mu} m - b_{\mu} n = \pm 1$ and such that
  the function $(x,y) \mapsto x^{a_{\mu}} y^{b_{\mu}} / x^ny^m$ is strictly meromorphic. In turn, up to a multiplicative
  function $u_1 (x^ny^m)$, the vector field $Y$ has the form
  $$
  Y = (x^n y^m)^{k_1} x^{a_{\mu}} y^{b_{\mu}}
  [mx \partial /\partial x - ny  \partial /\partial y + x^{-a_{\mu}} y^{-b_{\mu}} u_2 (x^ny^m) [b x \ddx - ay \ddy]]
  $$
  with $u_1, u_2$ holomorphic functions of a single variable satisfying the following conditions: $u_1(0) =1$ and the
  $(x,y) \mapsto x^{-a_{\mu}} y^{-b_{\mu}} u_2 (x^ny^m)$ is holomorphic of order at least~$1$ at the origin.

\end{itemize}
\end{teo}

The proof of Main Theorem can now be obtained as follows.

\begin{proof}[Proof of Main Theorem]
First we remind the reader of the basic observation pointed out at the beginning of Section~\ref{morelikebasics}.
Namely assume that $X$ and $Y$ form a rank~$2$ system of commuting holomorphic vector fields. Then every (holomorphic/meromorphic)
vector field $Z$ commuting with $X$ has the form $Z = fX + gY$ where $f,g$ are first integrals of $X$. Also, owing
to Proposition~\ref{solving-centralizer_elliptic}, the elliptic cases
in Table~\ref{TheTable} can be ruled out from the discussion.

Let us then begin the discussion with the case of parabolic vector fields.
Here Proposition~\ref{non_invertible_function} allows us to consider coordinates where the invertible function
appearing in Table~\ref{TheTable} is actually constant. More precisely we can assume that $X$ is given
by $x^2\ddx - y (nx-(n+1)y)\ddy$, with $n\in \mathbb{N}$, or by $(y-2x^2) \ddx - 2xy \ddy$. The corresponding
first integrals are $x^{n+1}y/(x-y)$ and $y^2/(y-x^2)$. Clearly they are
{\it strictly meromorphic}. Furthermore, we already know
vector fields $Y$ forming a rank~$2$ system of commuting semicomplete vector fields with $X$ in each of these
cases. Namely we can choose $Y = x^ny^2 \ddy$ - when $X$ is a quadratic vector field - and $Y = y(x \ddx + y\ddy)$ when
$X$ is nilpotent. Thus the general form of a vector field commuting with $X$ would be
$$
fX + gY
$$
where $f$ and $g$ are first integrals of $X$. These first integrals are however functions of the above indicated
meromorphic first integrals of~$X$ and thus they only way for the vector field $fX + gY$ to be holomorphic is to have
$f$ and $g$ constant. The corresponding statement in Main Theorem follows at once.

The same argument can be applied
when $X$ has the form indicated in item~11 of Table~\ref{TheTable}, with $n \in \N^{\ast}$. Again
Proposition~\ref{solving-centralizer_elliptic} allows us to assume that $X$ is given by $y(nx \ddx
+ y \ddy)$ in suitable coordinates. In particular, $X$ admits $x/y^n$ as a meromorphic first integral.
An example of holomorphic vector field forming a rank~$2$ system of commuting vector fields with
$X$ is provided by $Y = y^n \ddx$.
The main difference with the case of parabolic vector fields
arises from the fact that $Y$ multiplied by $x/y^n$
is again holomorphic. Namely $Y$ becomes $x \ddx$. However the vector field $x \ddx$ can be ruled out from the discussion
since it has one eigenvalue different from zero at the origin. Nonetheless, if in addition $n=1$, then $X$ multiplied
by $x/y$ yields the vector field
$$
x(x \ddx + y \ddy)
$$
which is again holomorphic, semicomplete, and has both eigenvalues at the origin equal to zero.
Thus the general form of a holomorphic vector field commuting with $X = y(x \ddx + y \ddy)$ and having both
eigenvalues at the origin equal to zero is
\begin{equation}
c_1 y \ddx + c_2 x(x \ddx + y \ddy) \label{justonecomment}
\end{equation}
with $c_1, c_2 \in \C$ as in Main Theorem. Finally, as pointed out in the introduction, all vector fields in~(\ref{justonecomment})
are semicomplete. In fact, for $c \neq 0$, all the vector fields of the form $y\ddx + cx (x\ddx + y\ddy)$ are conjugate
to the nilpotent parabolic vector field $(y-2x^2) \ddx - 2xy \ddy$.

The remaining possibilities for the foliation associated with $X$ (resp. $Y$) is either to be regular or to have a Martinet-Ramis
singularity at the origin. If one of them is regular, then the result follows from the first part of Theorem~\ref{harder_cases}.
Finally if they both have Martinet-Ramis singularities the second part of Theorem~\ref{harder_cases} implies the statement of
Main Theorem.
\end{proof}

In terms of solving the centralizer problem as stated in this paper it only remains to prove
Theorem~\ref{harder_cases}. The corresponding proof will take up the remainder of this section and the whole of next section.
In fact, the case in which the foliation associated with one of the vector fields, $X$ or $Y$, is regular
will be treated below whereas the case in which both vector fields have associated foliations with
Martinet-Ramis singular points will be the object of the next section.

In the sequel, we assume that $X$ and $Y$ are as in Theorem~\ref{harder_cases} and that the foliation associated with
$X$ is regular at the origin. Owing to Table~\ref{TheTable}, we can then set
\begin{equation}
X = y^k F(x,y) \ddx \label{regular_foliations_forX}
\end{equation}
with $a \in \N$ and where $F(x,y)$ has one of the following three forms: constant (equal to~$1$), $F(x,y) =x$,
or $F(x,y) = x^2 + g_1(y) x + g_2(y)$ where $g_1$ and $g_2$ are holomorphic functions on $(\mathbb{C}^2,0)$
satisfying $g_1(0) = g_2(0) = 0$. In the same coordinates, we let $Y = A \ddx + B \ddy$. The condition
$[X,Y]=0$ immediately implies that $B$ depends only on $y$, i.e. $B =B(y)$. In particular $B (y) = y^2 + {\rm h.o.t.}$
(cf. the proof of Lemma~\ref{commuting_with_regularfoliation}). Assume now that the foliation associated
with $Y$ {\it is not regular}\, at the origin. Then the combination
of Lemma~\ref{commuting_with_regularfoliation} and Proposition~\ref{non_invertible_function} ensures the existence
of $n \in \N^{\ast}$ such that $Y$ is conjugate to one of the following vector fields:
\begin{itemize}
  \item $x^2\ddx - y (nx-(n+1)y)\ddy$;
  \item $y(nx \ddx + y \ddy)$;
  \item $yf(y)(nx \ddx - y \ddy)$, where $f(0) \neq 0$.
\end{itemize}
In the first two cases, the centralizer of the corresponding vector field $Y$ was already worked out
in detail (cf. the above proof of Main Theorem) and it follows that Theorem~\ref{harder_cases} holds in the present context.
The third case where $Y$ is conjugate to $yf(y)(nx \ddx - y \ddy)$, with $f(0) \neq 0$ and $n \in \N^{\ast}$,
requires a few additional comments. First, there follows from the proof of
Lemma~\ref{eliminatinginvertiblefactor-f} in Section~\ref{resonantfoliations} that, again, $f$ can be made constant.
Thus $Y$ is actually of the form $y(nx \ddx - y \ddy)$ and $X$ is a semicomplete vector field commuting with $Y$
and inducing a regular foliation on a neighborhood of $(0,0) \in \C^2$. An example of vector field forming a rank~$2$ system
of commuting vector fields with $Y = y(nx \ddx - y \ddy)$ is provided by $y^nx^2 \partial /\partial x$.
Thus the general vector field commuting with $Y$, has the form
\begin{equation}
f_1 y^n x^2 \partial /\partial x + f_2  y(nx \ddx - y \ddy) \, \label{justabove-formula}
\end{equation}
where $f_1, f_2$ are first integrals of $Y$ and hence functions of $xy^n$. Precisely $f_1, f_2$ are functions of
a single variable $z$ and by $f_1$ (resp. $f_2)$ above, we actually mean the composition $(x,y) \mapsto f_1 (xy^n)$
(resp. $(x,y) \mapsto f_2 (xy^n)$).

In particular, $X$ must have the form
indicated in~(\ref{justabove-formula}). However the foliation associated with
$X$ must be regular meaning that $f_1 y^nx^2$ must divide the vector field $X$ itself which, in turn, means
that $f_1 y^nx^2$ must divide $yf_2$. Thus we have
\begin{equation}
X = f_1 y^nx^2 \left[ \frac{\partial}{\partial x} + \frac{yf_2}{f_1 y^nx^2} \left( nx \frac{\partial}{\partial x}
- y \frac{\partial}{\partial y} \right) \right] \, , \label{CorrectingtheformofX-regular}
\end{equation}
where the quotient $yf_2 / f_1 y^nx^2$ is a holomorphic function. Note that, in principle, $f_1$ might have a pole
of order~$1$ at $0 \in \C$. This possibility however is ruled out by the fact that the eigenvalues of $X$ at the
origin must be zero. Thus $f_1$ is actually holomorphic at $0 \in \C$. Moreover, $X$ still must be semicomplete
and this implies that $f_1 (0) \neq 0$ so that we can assume $f_1(0) =1$. Note that this implies that $f_2$ is holomorphic
as well and, as function of a single variable, it also satisfies $f_2 (0) = f_2' (0) =0$.
Whereas $X$ must admit the normal form of Formula~(\ref{CorrectingtheformofX-regular}) as indicated above, it is not
clear at this point whether or not all vector fields given by Formula~(\ref{CorrectingtheformofX-regular}) are, in fact,
semicomplete. This, however, is precisely the content of Lemma~\ref{addedlater-lemmaXSiegel} below:

\begin{lema}
\label{addedlater-lemmaXSiegel}
Let $X$ be given by
\begin{equation}
X = g_1 (xy^n) y^nx^2 \left[ \frac{\partial}{\partial x} + y^{n+1} g_2 (xy^n)  \left( nx \frac{\partial}{\partial x}
- y \frac{\partial}{\partial y} \right) \right] \, , \label{CorrectingtheformofX-regular-lemmaversion}
\end{equation}
where $g_1$, $g_2$ are holomorphic functions of a single variable with $g_1 (0) =1$. Then the vector field $X$
is semicomplete if and only if $g_1' (0) = g_2 (0) =0$.
\end{lema}

Summarizing what precedes, in order to understand the case of pairs of vector fields $X$ and $Y$
as in Theorem~\ref{harder_cases} such that the foliation associated with $X$
is regular, we can assume now that the foliation associated with $Y$ is regular as well.
This condition will be assumed to hold in the remainder of the section. Recall that $Y = A \ddx + B (y) \ddy$
with $B(0) = B'(0) =0$. Note that $A (0,0) = 0$ since
the vector field $Y$ has a singular point at the origin (actually this singular point must also have both eigenvalues
equal to zero). Therefore
$A$ is divisible by $y$. Thus one of the following possibilities must hold:
\begin{enumerate}
       \item $Y = A \ddx + B (y) \ddy =  y(a (x,y) \ddx + b(y) \ddy)$ with $a(0,0) \neq 0$ and $b(0) =0$;
       \item $Y = A \ddx + B (y) \ddy =  y^2 (a (x,y) \ddx + b(y) \ddy)$ where necessarily $b(0) \neq 0$.
\end{enumerate}
Next we have:

\begin{lema}
\label{finishing_job_regularfoliation-1}
Assume that $X$ and $Y$ are holomorphic vector fields as in the statement of Theorem~\ref{harder_cases}.
Assume also that $X =y^k F(x,y) \ddx$ whereas $Y$ has the above indicated form~(2), i.e.
$Y = y^2 (a (x,y) \ddx + b(y) \ddy)$ with $b(0) \neq 0$. Then the coordinates $(x,y)$ can be chosen so that
$$
X = x^2 \ddx \; \; \; \; {\rm and} \; \; \; \; Y = y^2 \ddy \, .
$$

\end{lema}

\begin{proof}
Note that the foliation associated with the vector field $Y$ is regular at the origin and it also transverse
to the (regular) foliation associated with $X$. Therefore there exist local coordinates $(u,v)$ where the
foliation associated with $X$ is given by $\partial /\partial u$ and the foliation associated with $Y$
is given by $\partial /\partial v$. In fact, in these $(u,v)$-coordinates, we have
$$
X = v^k f(u,v) \partial /\partial u \; \; \; \; {\rm and} \; \; \; \;
Y = v^2 g(u,v) \partial /\partial v \, .
$$
The condition $[X,Y]=0$ now amounts to having
$$
\frac{\partial (v^k f(u,v))}{\partial v} =0 \; \; \; {\rm and} \; \; \;
\frac{\partial (v^2 g(u,v))}{\partial u} =0 \, .
$$
In other words, the function $v^k f(u,v)$ does not depend on~$v$ and the function $v^2 g(u,v)$ does not
depend on~$u$. Thus we must have
$k=0$, $f = f(u)$ with $f(u) = u^2 + {\rm h.o.t.}$, and $g=g(v)$. Furthermore $g(0) \neq 0$ since $b(0) \neq 0$
(or more directly $Y$ is semicomplete). Finally
since again these vector fields must be semicomplete, Lemma~\ref{one-dimensional_cases} implies that $u$
and $v$ can independently be changed in coordinates $x$ and $y$ where the vector fields $X$ and
$Y$ have the indicated forms. The proof of the lemma is finished.
\end{proof}

Lemma~\ref{finishing_job_regularfoliation-2} summarizes the case in which the foliation
associated with the vector field $X$ is regular at the origin.

\begin{lema}
\label{finishing_job_regularfoliation-2}
Assume that $X$ and $Y$ are holomorphic vector fields as in the statement of Theorem~\ref{harder_cases}.
Assume also that the foliation associated with $X$ is regular at the origin.
Then, up to linear combination, there are coordinates where the pair $X$ and $Y$ takes on
one of the forms indicated in item~(A) of Theorem~\ref{harder_cases}.
\end{lema}

\begin{proof}
We keep the preceding notation. As previously seen, and up to proving Lemma~\ref{addedlater-lemmaXSiegel},
we only need to consider the case in which the foliation
associated with $Y$ is regular at the origin. Moreover, owing to Lemma~\ref{finishing_job_regularfoliation-1},
we can assume without loss of generality that $Y$ admits the form~(1) above,
i.e. $Y = y(a (x,y) \ddx + b(y) \ddy)$ with $a(0,0) \neq 0$ and $b(0) =0$ so that the axis $\{ y=0\}$
is invariant by the foliation associated with $Y$. Also $A = ya$ and $B = yb$.
Besides the fact that $B$ depends solely on the variable~$y$, the condition $[X,Y] =0$ also yields
\begin{equation}
y^k F(x,Y) \,  \left( \frac{\partial A}{\partial x} \right) - A (x,y) \, \left( \frac{\partial (y^k F)}{\partial x} \right)
-B (y) \, \left( \frac{\partial (y^k F)}{\partial y} \right) = 0 \, . \label{forcommutation_includingaffinestructure}
\end{equation}
Recall that the axis $\{ y=0\}$ is invariant by the foliation associated with $Y$ ($a(0,0) \neq 0$ and $b(0) =0$).
Whereas $Y$ vanishes identically
over $\{ y=0\}$, it induces an {\it affine structure}\, on this axis by means of the construction in Section~4 of
\cite{GuR}. The same applies for the vector field $X$ and, in this case, the resulting affine structure is compatible
with the translation structure obtained directly from the restriction of $X$ to this axis provided that $k \neq 0$.
The main issue here is that the affine structures
induced by $X$ and by $Y$ must coincide since $X$ and $Y$ commute. This last assertion follows essentially from
Equation~(\ref{forcommutation_includingaffinestructure}) above. Indeed, and whereas in the sequel we will argue from the
point of view of these affine structure as developed in \cite{GuR}, the reader unfamiliar with
this material can produce self-contained - if slightly {\it ad-hoc} - proofs for our claims by systematically
replacing $F$, $A$, and $B$ by their respective forms, dividing both sides in
Equation~(\ref{forcommutation_includingaffinestructure}) by the maximal common power of~$y$ and then setting $y=0$
in the remaining equation.

Note that the case $k=0$ is special in what follows only because it forces $F$ to take on the form
$F(x,y) = x^2 + g_1(y) x + g_2(y)$ (the origin must be a singular point of $X$ having both eigenvalues
equal to zero). This said, the affine structure induced by $X$ on a neighborhood of $0 \in \{ y=0 \}$ arises
from the vector field $F(x,0) \ddx$. On the other hand, the affine structure arising from $Y$ is determined by
the vector field $a(x,0) \ddx$ and, in particular, this latter affine structure is regular at $0 \in \{ y=0 \}$
since $a(0,0) \neq 0$. Therefore the affine structure associated with $F(x,0) \ddx$ must be regular as well and this
implies $F(0,0) \neq 0$ which, in turn, implies that $F$ must be constant equal to~$1$ (in particular $k$ is strictly
positive). Also, up to a multiplicative constant, we set $a(0,0) =1$ and $B(y) = \alpha y^2 + {\rm h.o.t.}$ for
some $\alpha \neq 0$.

Plugging the information $F =1$ and $k >0$ into Equation~(\ref{forcommutation_includingaffinestructure}),
we conclude that
$$
\frac{\partial a}{\partial x} = \alpha k + r(y)
$$
where $r(y) = B (y)/y^2 - \alpha$. Since $a(0,0) =1$, there follows that
$$
a (x,y) = 1 + \alpha k x + xr(y) + s(y)
$$
where $s$ is a holomorphic function satisfying $s(0) =0$. Thus it only remains to check that the resulting
vector field $Y$ is, in fact, semicomplete on a neighborhood of $(0,0) \in \C^2$. For this, we will look for new coordinates
where the foliation {\it associated with $Y$}\, is given by horizontal lines.

Let us thus consider the change of coordinates $(\overline{x}, \overline{y}) \mapsto
(\overline{x}, u(\overline{x}, \overline{y})) = (x,y)$, where $u$ is a holomorphic function,
satisfying the following conditions:
\begin{itemize}
  \item $u(0, \overline{y}) = \overline{y}$.
  \item The change of coordinates send horizontal lines $\{ \overline{y} = {\rm cte} \}$ to leaves
  of the foliation associated with $Y$.
\end{itemize}
It straightforward to check that the vector field $Y$ in the coordinates $(\overline{x}, \overline{y})$
becomes (see proof of Lemma~\ref{addedlater-lemmaXSiegel} for details)
$$
Y = \overline{y} (1 + \beta (\overline{x}, \overline{y})) \partial /\partial \overline{x} \, ,
$$
where $\beta$ is a holomorphic function vanishing at the origin. Since $\overline{y}$ is constant over the horizontal
lines, $Y$ is semicomplete on a neighborhood of the origin if and only if the vector field
$$
(1 + \beta (\overline{x}, \overline{y})) \partial /\partial \overline{x}
$$
is so. This last vector field is however clearly semicomplete since it is regular at
the origin and hence admits a flow-box coordinate defined around $(0,0) \in \C^2$.
Lemma~\ref{finishing_job_regularfoliation-2} is proved.
\end{proof}

We close this section with the proof of Lemma~\ref{addedlater-lemmaXSiegel}

\begin{proof}[Proof of Lemma~\ref{addedlater-lemmaXSiegel}]
Note that the foliation associated with $X$ is again regular at the origin and hence we can again choose coordinates
$(\overline{x}, \overline{y}) \mapsto (\overline{x}, u(\overline{x}, \overline{y})) = (x,y)$
satisfying the conditions:
\begin{itemize}
  \item $u(0, \overline{y}) = \overline{y}$.
  \item The change of coordinates send horizontal lines $\{ \overline{y} = {\rm cte} \}$ to leaves
  of the foliation associated with $Y$.
\end{itemize}
The vector field $X$ then becomes
$$
X = \overline{y}^n \overline{x}^2 (1 + \beta (\overline{x}, \overline{y})) \partial /\partial \overline{x}.
$$
Hence $X$ is semicomplete if and only if the vector field
$$
\overline{x}^2 (1 + \beta (\overline{x}, \overline{y})) \partial /\partial \overline{x}
$$
is so. In turn, this means that for every (small) fixed value $\overline{y}_0$ of $\overline{y}$
the $1$-dimensional vector field
$\overline{x}^2 (1 + \beta (\overline{x}, \overline{y}_0)) \partial /\partial \overline{x}$ is semicomplete.
In particular, its residue at $0 \in \C$ must be zero (Lemma~\ref{one-dimensional_cases}). By setting
$\beta (\overline{x}, \overline{y})) = \sum_{j=0}^{\infty} a_j (y) x^j$, this last condition means that
$a_1 (y)$ is identically zero. In other words, the Taylor series of $u$ at $(0,0)$ contains no monomials
of the form $\overline{x} \, \overline{y}^k$. Conversely, if this condition holds then the
corresponding $1$-dimensional vector field
is conjugate to (a multiple of) $\overline{x}^2 \partial /\partial \overline{x}$ and thus semicomple on some neighborhood
of $0 \in \C$. The domain of definition of the mentioned conjugation can, however,
be made of uniform size as $\overline{y}_0 \rightarrow 0$ so that
$X$ is semicomplete as desired.

The rest of the proof consists of showing that $\beta$ satisfies the above condition if and only if the initial
functions $g_1$ and $g_2$ are as indicated in the statement of Lemma~\ref{addedlater-lemmaXSiegel}. For this we proceed
as follows.

Set
\begin{equation}
\Theta (x,y)  =  \frac{-y^{n+1} g_2 (xy^n)}{1+ nxy^n g_2 (xy^n)}
= -y^{n+1} g_2 (xy^n) [1 - nxy^n g_2 (xy^n) + \cdots] \, . \label{Ultimateform_Theta}
\end{equation}
On the other hand the condition that the leaves of the initial foliation are given by $\{ \overline{y} = {\rm cte} \}$
amounts to the equation
\begin{equation}
\frac{\partial u}{\partial \overline{x}} (\overline{x}, \overline{y}) = \Theta (\overline{x},
u (\overline{x}, \overline{y})) \, .\label{Equation_for_u}
\end{equation}
Next let $u (\overline{x}, \overline{y}) = \sum_{i=1}^{\infty} c_i (\overline{x}) \overline{y}^i$.
The condition $u(0, \overline{y}) = \overline{y}$ then implies $c_1(0) =1$ and $c_i (0) =0$ for $i \geq 2$.
Plugging $u$ into Equation~(\ref{Equation_for_u}), dividing by $\overline{y}$ and then
setting $\overline{y} =0$, we conclude that $c_1'(x) =0$ so that $c_1 (x)$ is constant equal to~$1$.
Thus $u (\overline{x}, \overline{y}) = \overline{y} (1 + \sum_{i=2}^{\infty} c_i (\overline{x}) \overline{y}^{i-1})$
and the coefficients functions $c_i (x)$ can recursively be computed. In particular if
$g_2 (0) \neq 0$, then all coefficients $c_i$ satisfies $c_i'(0) =0$ (and some of them are identically zero).
Thus $u$ does not have monomials of the form $\overline{x} \, \overline{y}^k$ in its Taylor series.
Conversely, if $g_2 (0) \neq 0$,
then $c_{n+1}' (0) \neq 0$ and hence the Taylor series of $u$ contains the monomial $\overline{x} \, \overline{y}^{n+1}$.

Finally in coordinates $(\overline{x}, \overline{y})$, the vector field $X$ becomes
$$
\overline{x}^2 g_1 [\overline{x} (u (\overline{x}, \overline{y}))^n]  (u (\overline{x}, \overline{y}))^n
\partial /\partial \overline{x}
$$
so that
$$
1 + \beta (\overline{x}, \overline{y}_0) =
g_1 [\overline{x} (u (\overline{x}, \overline{y}))^n]  (u (\overline{x}, \overline{y}))^n \, .
$$
Taking into account the formula
$u (\overline{x}, \overline{y})) = \overline{y} (1 + \sum_{i=2}^{\infty} c_i (\overline{x}) \overline{y}^{i-1})$
and the fact that $g_1 (0) =1$, it is easy to check that the absence of monomials of the form $\overline{x} \,
\overline{y}^k$
in the Taylor series of $1 + \beta (\overline{x}, \overline{y}_0)$ is equivalent to the following pair
of conditions:
\begin{enumerate}
  \item $g_1' (0) =0$;

  \item The Taylor series of $u$ at $(0,0)$ contains no monomial of the form $\overline{x} \, \overline{y}^k$.
\end{enumerate}
As previously seen, condition~(2) is satisfied if and only if $g_2 (0) =0$. The lemma follows.
\end{proof}

\section{Vector fields whose associated foliation is of Martinet-Ramis type}\label{resonantfoliations}

To complete the proof of Main Theorem, it only remains to discuss the case of holomorphic vector fields
$X$ and $Y$ forming a rank~$2$ systems of commuting semicomplete vector fields under the following
additional condition: the foliation associated with the vector field $X$ and the foliation associated with the
vector field $Y$ both have Martinet-Ramis singular points at the origin. The remainder of this paper
will be devoted to understanding this situation.

We begin with a general simple lemma.

\begin{lema}
\label{commutinginduced}
Let $\fol_Y$ be the foliation associated with a holomorphic vector field $Y$. Suppose
that $X$ is another holomorphic vector field that commutes with $Y$. Then the local flow of
$X$ induces a (local) $1$-parameter group of automorphisms in the leaf space of
$\fol_Y$.
\end{lema}

\begin{proof}
First we would like to point out that the leaf space of $\fol_Y$
was not yet endowed with any particular structure. In this sense the meaning of the
above statement is precisely the following: suppose that we are given points $x_1, x_2$,
belonging to the same leaf of $\fol_Y$, and consider $T \in \C$ such that $\Phi (T,x_1)$,
$\Phi (T,x_2)$ are both defined, where $\Phi$ denotes the local flow of $X$. Then
$\Phi (T,x_1)$, $\Phi (T,x_2)$ also belong to the same leaf of $\fol_Y$. To check the
claim we note that it has a local nature. Therefore it suffices to prove it for $x_1$
close to $x_2$ and for $T$ small in norm. If $\Psi$ denotes the local flow of $Y$, the
preceding allows us to suppose that $x_2 = \Psi (t, x_1)$ for small $t \in \C$. The
commutativity of $X, \, Y$ then provides
$$
\Phi (T ,x_2) = \Phi (T, \Psi (t, x_1)) = \Psi (t, \Phi (T, x_1))
$$
as long as all the terms in this equation are well defined. We then conclude that
$\Phi (T,x_1)$, $\Phi (T,x_2)$ belong to the same leaf of $\fol_Y$ as desired.
\end{proof}

Throughout this section $X$ and $Y$ will be holomorphic vector fields defined around $(0,0) \in \C^2$ and
satisfying all of the conditions below:
\begin{enumerate}
  \item $X$ and $Y$ form a rank~$2$ system of commuting semicomplete vector fields;

  \item The eigenvalues of $X$ and of $Y$ at $(0,0) \in \C^2$ are all equal to zero;

  \item The foliation $\fol_X$ associated with $X$ and the foliation $\fol_Y$ associated with $Y$ both
  have Martinet-Ramis singular points at $(0,0) \in \C^2$.
\end{enumerate}

Next we have a slightly more specific lemma

\begin{lema}
\label{commonseparatrices}
Let $X,Y$ be as above. Then $\fol_X, \, \fol_Y$ share the same separatrices.
\end{lema}

\begin{proof}
Since the origin is a Martinet-Ramis singular point of $\fol_X$ (resp. $\fol_Y$),
there follows that $\fol_X$ (resp. $\fol_Y$) possesses exactly two irreducible separatrices,
see \cite{mattei}  or \cite{MR-ENS}.
In addition these separatrices are smooth and mutually transverse. Let then $S_{1,X}$ be one of
the separatrices of $\fol_X$. Since the origin is a fixed point of the (semi-global) flow $\Phi_Y$ of $Y$,
there immediately follows that $\Phi_Y^T (S_{1,X})$ is still a germ of analytic curve passing through
$(0,0)$ which is invariant under the flow $\Phi_X$ of $X$. Up to choosing $T$ small enough, we conclude
that $\Phi_Y$ must preserve $S_{1,X}$. We then need to prove that $S_{1,X}$ is a separatrix of $\fol_Y$ as well.

Assume aiming at a contraction that $S_{1,X}$ is not a separatrix for $\fol_Y$. Since $S_{1,X}$ must be invariant by
$\Phi_Y$, it follows that $Y$ has zeros over $S_{1,X}$. In view of Table~\ref{TheTable}, this last assertion implies
the existence of local coordinates $(x,y)$ where
$$
Y = x^n y^n (x-y) f [x \ddx - y\ddy]
$$
with $f(0,0) \neq 0$, $n \in \mathbb{N}$, and $S_{1,X} = \{ x=y\}$. Conversely, $X$ must also have zeros over one of
the separatrices $S_{1,Y}$ of $\fol_Y$ which, in the above coordinates, can be chosen as $\{ x=0\}$ (and $S_{1,Y} =
\{ x=0\}$ is not a separatrix for the foliation associated with $X$).
Since the line $S_{1,X} =\{ x=y\}$ is transverse
to $\fol_Y$ (at generic points), it parameterizes an open set of the leaf space of $\fol_Y$. Thus $X$ cannot vanish
identically over $S_{1,X} = \{x=y\}$, otherwise the automorphism induced by $X$ on the leaf space of $\fol_Y$ would be trivial
which, in turn, would contradict the fact that $X$ and $Y$ are linearly independent at generic points (cf. condition~(1)).
In other words, the vector field $X$ does not vanish identically on one of the separatrices of its associated
foliation. Table~\ref{TheTable} then ensures that
the zero-set of $X$ consists solely of the line $\{ x=0\}$. Similarly the zero-set of $Y$ is reduced to the line
$\{ x=y\}$ so that $n=0$. Thus the first non-zero (i.e. quadratic) homogeneous components of the Taylor series of $X$
and of $Y$ are respectively given by
$$
x [(x-2y) \ddx - y\ddy] \; \; \; \; {\rm and} \; \; \; \; (x-y)[ x\ddx -y \ddy] \; .
$$
A direct inspection however shows that these quadratic vector fields do not commute and this contradicts the fact
that $[X,Y] =0$. The lemma is proved.
\end{proof}

Lemma~\ref{commonseparatrices} can be made slightly more accurate as follows:

\begin{lema}
\label{existence_firstcoordinates}
Assume that $X$ and $Y$ satisfy conditions~(1), (2), and~(3) above. Then there are coordinates $(x,y)$ where
the separatrices of $\fol_X$ and of $\fol_Y$ are given by the union of the axes $\{ x=0\}$ and $\{ y=0\}$.
Furthermore the zero-set of either $X$ or $Y$ is contained in the union of these (common) separatrices.
\end{lema}

\begin{proof}
The existence of coordinates $(x,y)$ such that the coordinates axes $\{ x=0\}$ and $\{ y=0\}$ correspond
to the separatrices of $\fol_X$ (resp. $\fol_Y$) is an immediate consequence of Lemma~\ref{commonseparatrices} combined
with the well-known general normal form of foliations having Martinet-Ramis singular points (cf. \cite{mattei}  or \cite{MR-ENS}).

To show that the zero-set of $X$ (resp. $Y$) is contained in the union of the corresponding coordinates
axes, it suffices to rule out the possibility of having $X$ (resp. $Y$) as in item~13 of
Table~\ref{TheTable}. However, if $X$ were as in item~13 of Table~\ref{TheTable}, it would vanish identically
over a curve $C$ which is not a separatrix for $\fol_Y$. As previously seen, this implies that the automorphism
induced by $X$ in the leaf space
of $\fol_Y$ is trivial contradicting the condition~(1).
This establishes the lemma.
\end{proof}

What precedes can be summarized by claiming the existence of local coordinates $(x,y)$ where
the vector field $X$ takes on the form
\begin{equation}
X = x^a y^b f(x,y) [mx (1 + {\rm h.o.t.}) \partial /\partial x - ny (1 + {\rm h.o.t.}) \partial /\partial
y] \label{generalnormalform_X}
\end{equation}
where $m,n \in \N^{\ast}$, $f(0,0) \neq 0$, and $a,b \in \N$ with at least one between $a$ and $b$
different from zero (since the eigenvalues
of $X$ at $(0,0)$ are both zero). Similarly, we also have
\begin{equation}
Y = x^{a^*}y^{b^*} f^{\ast} [m^* x (1 + {\rm h.o.t.}) \partial /\partial x - n^* y (1 + {\rm h.o.t.}) \partial /\partial y]
\label{generalnormalform_Y} \, .
\end{equation}
Here again $f^{\ast} (0,0) \neq 0$, $m^{\ast},n^{\ast} \in \N^{\ast}$, and
$a^{\ast},b^{\ast} \in \N$ with at least one of them different from zero.

Recall also that $m,n,a,b$ (resp. $m^{\ast},n^{\ast}, a^{\ast},b^{\ast}$) are bound by the relation
$am-bn \in \{-1, 0, 1\}$ (resp. $a^{\ast}m^{\ast}-b^{\ast} n^{\ast} \in \{-1, 0, 1\}$). Furthermore,
$\fol_X$ (resp. $\fol_Y$) is necessarily linearizable if $am -bn \neq 0$ (resp.
$a^{\ast}m^{\ast}-b^{\ast} n^{\ast} \neq 0$).

Finally note that $\fol_X$ and $\fol_Y$ share the same eigenvalues if and only if $m/m^{\ast} = n/n^{\ast}$.
In fact, the eigenvalues of a foliation are defined only up to multiplicative constants.


For $X$ (resp. $Y$) as above, let $X^H$ (resp. $Y^H$) denote the first non-zero homogeneous component of the
Taylor series of $X$ (resp. $Y$) at the origin. In other words, up to multiplicative constants, we set
\begin{equation}
X^H = x^ay^b [mx \partial /\partial x - ny  \partial /\partial y] \; \; {\rm and} \; \;
Y^H = x^{a^*}y^{b^*} [m^* x \partial /\partial x - n^* y  \partial /\partial y] \, .
\label{X-H_and_Y-H}
\end{equation}
Now we have:

\begin{lema}
\label{lemmaadded1}
Assume that $X, \, Y$ satisfy conditions~(1), (2), and~(3) at the beginning of this section.
If the eigenvalues of $\fol_X, \, \fol_Y$ do not coincide then $x^ay^b$ is a first integral for
the linear foliation induced by $Y^H$. Similarly, $x^{a^*}y^{b^*}$
is a first integral for the linear foliation induced by $X^H$.
\end{lema}

\begin{proof}
Since $X$ and $Y$ commute, the homogeneous vector fields $X^H, \, Y^H$ must commute as well. Therefore,
we have
\begin{equation}\label{cond_vp}
\begin{cases}
b - b^* = \frac{m^*}{n^*} a - \frac{m}{n} a^* \, , \\
a - a^* = \frac{n^*}{m^*} b - \frac{n}{m} b^* \, .
\end{cases}
\end{equation}
Solving the second equation for $a$ and substituting it in the first one, we obtain
\[
\frac{nm^* - mn^*}{mn^*} b^* = \frac{nm^* - mn^*}{nn^*} a^* \, .
\]
This equation is verified if and only if
\[
nm^* - mn^* = 0 \quad \text{or} \quad b^* = \frac{m}{n} a^* \, .
\]
Assuming that the eigenvalues of $X^H, \, Y^H$ do not coincide, we must have the second situation, i.e.
$b^* = ma^*/n$. Substituting now the expression of $b^*$ in the second equation, the system~(\ref{cond_vp})
becomes equivalent to
\begin{equation}
\begin{cases}\label{cond_vp_simp}
b^* = \frac{m}{n} a^* \, , \\
b = \frac{m^*}{n^*} a \, .
\end{cases}
\end{equation}
By assumption $a$ and $b$ (resp. $a^*$ and $b^*$) cannot simultaneously be equal to zero. Thus the above
equations imply that all the constants $a, b, a^{\ast}, b^{\ast}$ are
different from zero. Therefore the preceding conditions can be reformulated as
\begin{equation}\label{mnrs}
\frac{b}{a} = \frac{m^*}{n^*} \quad {\rm and} \quad \frac{b^*}{a^*} = \frac{m}{n} \, .
\end{equation}
The lemma follows.
\end{proof}

An immediate consequence of what precedes reads as follows.

\begin{coro}\label{corol_linearizable}
If the eigenvalues of $\fol_X, \, \fol_Y$ do not coincide, then $\fol_X, \, \fol_Y$ are linearizable
(though not necessarily linearizable in the same coordinate).
\end{coro}

\begin{proof}
Let $X, \, Y$ as above and assume that $am - bn = 0$. Then we must have
$m/n = b/a = m^*/n^*$ so that $\fol_X$ and $\fol_Y$ have the same eigenvalues at the origin. Thus we must
have $am - bn \in \{-1, 1\}$ and hence $\fol_X$ is linearizable. An analogous argument shows that $\fol_Y$
must be linearizable as well.
\end{proof}

Lemma~\ref{eliminatinginvertiblefactor-f} below provides us with sufficient conditions to ensure that the invertible
multiplicative functions $f$ becomes constant in suitable coordinates. This lemma will play an important role
in the subsequent discussion.

\begin{lema}
\label{eliminatinginvertiblefactor-f}
Consider vector fields $X, \, Y$ satisfying conditions~(1), (2), and~(3) at the beginning of this section.
Assume that the foliation $\fol_X$ is linearizable. Then there are local coordinates $(x,y)$ where $X$ becomes
$$
x^a y^b [mx \partial /\partial x - ny  \partial /\partial y]
$$
with $am-bn \in \{-1, 0, 1\}$.
\end{lema}

\begin{proof}
Under the assumption that $\fol_X$ is linearizable, the vector field $X$ can be written as
$X=x^a y^b f(x,y)[mx \partial /\partial x - ny  \partial /\partial y]$, with $am-bn \in \{-1, 0, 1\}$,
and where $f$ is a holomorphic function satisfying $f(0,0) \neq 0$. The statement means that
the coordinates $(x,y)$ can be chosen so as to have, in addition, the function $f$ equal to a constant.

The existence of the desired coordinates $(x,y)$ will be established by constructing
a local diffeomorphism taking $X$ to the vector
field $\mathcal{Z} = x^a y^b [mx \partial /\partial x - ny  \partial /\partial y]$. In the sequel, the linear foliation
associated with $\mathcal{Z}$ will be denoted by $\fol$.

Consider a local section $\Sigma$ through a point $x_0$ of $\{y=0\}$ which is transverse to both $\fol_X$ and
$\fol$. The local holonomy
maps arising from a small path contained in $\{y=0\}$ and encircling the origin are finite in
both cases. In fact, the holonomy map arising from the foliation $\fol_X$ is conjugate to the
holonomy map arising from $\fol$. In particular, the fundamental group of the leaves of $\fol_X$ (resp. $\fol$)
is cyclic infinite. This allows us to talk about a {\it period}\, for the vector field $X$ (resp. $Z$)
as follows: consider a path contained in $\{y=0\}$ and
winding around the origin the order of the corresponding holonomy map (namely $n$ times). Given
a leaf $L$ of $\fol_X$ (resp. $\fol$), denote by $c_L$ the lift of the mentioned
path in $L$. Clearly the path $c_L$ is a generator of the fundamental group of $L$. The {\it period}\, of $L$
with respect to $X$ (resp. $Z$) is then the value of the integral of $dT_L$ over $c_L$, where
$dT_L$ stands for the time-form induced by $X$ (resp. $\mathcal{Z}$) on $L$. The reader will easily check
that these periods are always zero for semicomplete vector fields $X$ such that $am-bn \in \{-1, 1\}$.
Conversely they are non-zero if $am - bn =0$.

To prove the existence of a diffeomorphism taking $\mathcal{Z}$ to $X$, it suffices to consider the
particular case where $m=n=1$. In fact, the general case can be reduced
to $m=n=1$ by lifting the vector field through the ramified covering $(x,y) \mapsto (x^m,y^n)$.
The choice $m=n=1$ allows us to abridge notation since the
local holonomy map associated to $\fol_X$ (resp. $\fol$) coincides with the identity in this case.

Let us first consider the case in which the periods of the leaves of $\fol_X$ with respect to~$X$ are zero.
As pointed out above, this means that $am-bn = a-b \in \{-1, 1\}$. The condition $a-b \in \{-1, 1\}$ also
implies that the periods of the leaves of $\fol$ with respect to $Z$ are zero as well.
To construct the desired diffeomorphism taking $X$ to $\mathcal{Z}$, let us consider again the transverse section $\Sigma$.
Let $\Phi_X$ (resp. $\Phi_{\mathcal{Z}}$) denote the (semi-global) flow of $X$ (resp. $\mathcal{Z}$). Away from
$\{x=0\} \cup \{y=0\}$, the diffeomorphism $H$ is defined by the following rule: we choose $T$ such that
$\Phi_X (T, (x,y)) = (x_0, y_0) \in \Sigma$ (note that $T$ and $(x_0, y_0)$ are well defined since the
holonomy of $\fol_X$ is trivial and since the period of $X$ is zero); then set $H(x,y) = \Phi_{\mathcal{Z}} (-T,
\Phi_X (T, (x,y))$. Clearly $H$ is well defined and a straightforward verification shows that $H$ extends
to the separatrices $\{x=0\} \cup \{y=0\}$ so as to define a holomorphic diffeomorphism on a neighborhood
of $(0,0)$; cf. \cite{mattei}.

To finish the proof of the lemma, there remains to consider the case where not all the
periods of the leaves of $\fol_X$ with respect to~$X$ are zero. This implies that $am-bn = a-b =0$.
In particular, it is immediate to check that all the leaves of $\fol$ have the same non-zero period
with respect to $\mathcal{Z}$.

The crucial point of the argument consists of showing that all leaves of $\fol_X$ have the same period with
respect to $X$ as well. In fact, if these periods are always the same and up to multiplying $X$ by a constant,
we can assume they also coincide with the
(constant) period obtained from the leaves of $\fol$ with respect to $\mathcal{Z}$. At this point
the same construction used above yields again a well-defined diffeomorphism $H$ conjugating $X$ and $\mathcal{Z}$.
Indeed, with the preceding notation, the point $(x_0, y_0)$ is well defined since it only depends on the fact
that the local holonomy maps of the foliations in question are trivial. The value of the time ``$T$'' however is no longer
well defined. Nonetheless, two different values of $T$ differ by a multiple of the period and therefore
$H(x,y) = \Phi_{\mathcal{Z}} (-T, \Phi_X (T, (x,y))$ becomes independent of the choice of $T$. In other words,
$H$ is well defined and conjugates $X$ and $\mathcal{Z}$ as required.

Finally to show that
all the leaves of $\fol_X$ have the same period with respect to $X$, it suffices to use the fact that
the flow of the vector field $Y$ permutes the leaves of $\fol_X$. More precisely,
the time-$t$ diffeomorphism induced by $Y$ preserves $X$. Thus it conjugates the different ($1$-dimensional)
vector fields obtained by restricting $X$ to the leaves of $\fol$. Hence the periods of the leaves of $\fol_X$ must
be constant and the lemma follows.
\end{proof}

Albeit simple, the next lemma is also important for the proof of Theorem~\ref{harder_cases}.

\begin{lema}
\label{lemmaDependence1}
Let $X, \, Y$ be two semicomplete holomorphic vector fields satisfying conditions~(1), (2), and~(3) at the beginning of this section.
Assume also that $X$ is given by Formula~(\ref{generalnormalform_X}) and that $\fol_X$ is linearizable.
Then $am-bn \neq 0$.
\end{lema}

\begin{proof}
Assume aiming at a contradiction that $am-bn=0$. Since $\fol_X$ is linearizable, there exist coordinates
$(x,y)$ where $X$ is given by
$$
X = (x^n y^m)^k f (x,y) [mx \partx - ny \party] \, ,
$$
with $k \geq 1$ and $f$ holomorphic satisfying $f(0,0) \neq 0$. Note that in the sequel we can even dispense with
the use of Lemma~\ref{eliminatinginvertiblefactor-f} to set $f$ constant equal to~$1$. In fact, we observe that the period
of the vector field $X$ must be constant since it commutes with $Y$ (with $X$ and $Y$ forming
a rank~$2$ system of vector fields, i.e. $X$ and $Y$ are not linearly dependent everywhere). Hence to obtain the contradiction
it suffices to check that the periods of the vector field $X$ as above {\it do vary}\, with the leaf of $\fol_X$.
For this note the leaf $L$ through a point $(x_0 ,y_0)$ can be parameterized by $T \mapsto (x_0 e^{mT} , y_0 e^{-nT})$
where $T$ belongs to some domain $\Omega \subset \C$ with is invariant by vertical translations. The restriction of
$X$ to $L$ viewed in the coordinate $T$ is simply $(x_0^n y_0^m)^k f(x_0 e^{mT}, y_0 e^{-nT}) \, \partial /\partial T$.
Fixed $T_0 = \Re (T_0) + \imath \Im (T_0) \in \Omega$, the period of $X$ on $L$ is obtained by integrating
$dT / [(x_0^n y_0^m)^k f(x_0 e^{mT}, y_0 e^{-nT})]$ over the vertical segment going from $T_0$ to
$\Re (T_0) + \imath (2 \pi + \Im (T_0))$. Since $f(0,0) =1$, the value of this integral is clearly affected
by the choice of $(x_0, y_0)$ and hence of the leaf $L$. The lemma is proved.
\end{proof}

The final main ingredient in the proof of Theorem~\ref{harder_cases} is the following proposition:

\begin{prop}
\label{automaticlinearization}
Let $X, \, Y$ be two holomorphic vector fields satisfying conditions~(1), (2), and~(3) at the beginning
of the section. Then the foliation $\fol_X$ (resp. $\fol_Y$) associated with $X$ (resp. $Y$) is linearizable.
\end{prop}

Taking for granted Proposition~\ref{automaticlinearization}, the proof of Theorem~\ref{harder_cases}
can now be provided.

\begin{proof}[Proof of Theorem~\ref{harder_cases}]
Let $X$ and $Y$ be vector fields as in the statement of the theorem in question. In view of the discussion
in Section~\ref{invertiblefunctions_etc}, we can assume that the foliations associated with $X$ and with $Y$
have both Martinet-Ramis singular points at the origin, cf. in particular Lemma~\ref{finishing_job_regularfoliation-1}
and Lemma~\ref{finishing_job_regularfoliation-2}. In other words, $X$ and $Y$ satisfy conditions~(1), (2), and~(3)
at the beginning of the present section. Actually, we can assume without loss of generality that
conditions~(1), (2), and~(3) are satisfied by all non-trivial linear combination of $X$ and $Y$ as well.

According to Proposition~\ref{automaticlinearization}, the foliation $\fol_X$ (resp. $\fol_Y$) associated
with $X$ (resp. $Y$) is linearizable. In turn, the combination of Lemma~\ref{eliminatinginvertiblefactor-f} and
Lemma~\ref{lemmaDependence1} yields local coordinates $(x,y)$ where $X$ becomes
$$
X = x^a y^b [mx \partial /\partial x - ny  \partial /\partial y]
$$
with $am-bn \in \{-1, 1\}$, $a,b \in \N$. Let $k_1$ be the greatest non-negative integer for which $(x^n y^m)^{k_1}$
divides $x^ay^b$. Then we can set
$$
X = (x^n y^m)^{k_1} x^{a_{\mu}} y^{b_{\mu}} [mx \partial /\partial x - ny  \partial /\partial y] \, ,
$$
where $a_{\mu}, b_{\mu}$ are non-negative integers satisfying $a_{\mu} m - b_{\mu} n = am-bn = \pm 1$ and such that
the function $(x,y) \mapsto x^{a_{\mu}} y^{b_{\mu}} / x^ny^m$ is strictly meromorphic.
On the other hand, the vector field $b x \ddx - ay \ddy$ commutes with $X$ and satisfies
condition~(1) with $X$. Therefore the general form of a holomorphic vector field commuting with $X$ and having
all eigenvalues at the origin equal to zero is given by
\begin{equation}
f_1 (x^n y^m) (x^n y^m)^{k_1} x^{a_{\mu}} y^{b_{\mu}} [mx \partial /\partial x - ny  \partial /\partial y]
+ f_2 (x^n y^m) [b x \ddx - ay \ddy] \, , \label{General_form_MartinetRamis_CommutingVF}
\end{equation}
where $f_1$ and $f_2$ are functions of a single variable satisfying the following conditions:
\begin{itemize}
  \item $f_1$ is meromorphic at $0\in \C$ with order greater than or equal to $-k_1$;
  \item $f_2$ is holomorphic around $0\in \C$ and, in addition, $f_2 (0) =0$.
\end{itemize}
The vector field $Y$ must admit the form indicated in Formula~(\ref{General_form_MartinetRamis_CommutingVF}).
However the foliation $\fol_Y$ associated with $Y$ must have a Martinet-Ramis singular point at the origin.
On the other hand, we can assume that $X$ and $Y$ have the same order at the origin, since $X$ and $Y$ can be replaced
by any non-trivial couple of linear combinations of them. From this, it easily follows that
$\fol_X$ and $\fol_Y$ must have the same eigenvalues ($m,-n$) at the origin as well. In addition, $Y$ must then have
the form
\begin{eqnarray*}
Y  & = &  (x^n y^m)^{k_1} x^{a_{\mu}} y^{b_{\mu}} u_1(x^ny^m)
[mx \partial /\partial x - ny  \partial /\partial y + x^{-a_{\mu}} y^{-b_{\mu}} u_2 (x^ny^m) [b x \ddx - ay \ddy]] \\
 & = & x^a y^b u_1(x^ny^m)
[mx \partial /\partial x - ny  \partial /\partial y + x^{-a_{\mu}} y^{-b_{\mu}} u_2 (x^ny^m) [b x \ddx - ay \ddy]] \, ,
\end{eqnarray*}
with $u_1, u_2$ holomorphic functions of a single variable such that:
\begin{enumerate}
  \item $u_1(0) =1$;
  \item The map
$(x,y) \mapsto x^{-a_{\mu}} y^{-b_{\mu}} u_2 (x^ny^m)$ is holomorphic of order at least~$1$ at the origin
(since the eigenvalues of $Y$ are $m,-n$).
\end{enumerate}

The proof of Theorem~\ref{harder_cases} is now reduced to showing that a vector field $Y$ as above must be semicomplete
on a neighborhood of the origin. For this we begin by noticing that the foliation associated with $Y$ is also
linearizable for the roles of $X$ and $Y$ can be permuted. Yet, in this respect, we can be slightly more accurate
while also providing an elementary proof independent of the semicomplete character of~$Y$. Indeed, consider
the vector field
$$
Z = mx \partial /\partial x - ny  \partial /\partial y + x^{-a_{\mu}} y^{-b_{\mu}} u_2 (x^ny^m) [b x \ddx - ay \ddy] \, .
$$
Recall that a monomial $x^{k_1} y^{k_2}$ ($k_1 + k_2 \geq 2$)
in the Taylor series of $Z$ is said to be {\it resonant}\, if either $m = k_1 m - k_2 n$ or $-n =  k_1 m - k_2 n$.
In other words, these are of the form $x^{1+km}y^{kn}$ in the component $\partial /\partial x$ and
of the form $x^{km}y^{1+kn}$ in the component $\partial /\partial y$, where $k \in \N^{\ast}$. Since
$a_{\mu} m - b_{\mu} n = \pm 1$, there follows that the Taylor series of $Z$ contains no resonant monomials.
The standard power series procedure then ensures that $Z$ is linearizable, as follows by a simple induction
argument, by taking into account that the desired
change of coordinates has the form $(x_1, y_1) = (x_1 (1 + \zeta_1 (x_1, y_1)) ,
y_1(1 +\zeta_2 (x_1, y_1)))$ where the
monomials with non-zero coefficients in the Taylor series of
$\zeta_j$ ($j=1,2$) have all the form $(x^{a_{\mu}} y^{b_{\mu}})^{l_1} (x^n y^m)^{l_2}$, with $l_1, l_2 \in \N$
and $\zeta_1 (0,0) = \zeta_2 (0,0) =0$. Additional details can be found for example in \cite{sad_camacho}.

Summarizing what precedes, in coordinates $(x_1, y_1)$
the vector field $Y$ becomes
\begin{equation}
Y = (x_1^n y_1^m)^{k_1} x_1^{a_{\mu}} y_1^{b_{\mu}}  v(x_1, y_1)
[mx_1 \partial /\partial x_1 - ny_1  \partial /\partial y_1] \label{Finishing_with_vectorfield_Y}
\end{equation}
where $v$ is some holomorphic function satisfying $v(0,0) =1$.
To conclude that $Y$ as in Formula~(\ref{Finishing_with_vectorfield_Y}) is semicomplete, we proceed as follows.
Recall that the leaves of the foliation associated with $Y$ are parameterized by
$T \mapsto (x_0 e^{mT} , y_0 e^{-nT})$
where $T$ belongs to some domain $\Omega \subset \C$ with is invariant by vertical translations,
cf. Lemma~\ref{lemmaDependence1}. Arguing as in Lemma~\ref{lemmaDependence1} (or actually with the
affine structure of \cite{GuR} and its relation with monodromy), we see that
the periods of $Y$ over the leaves of the foliation will vary with the leaf unless they are all equal to
zero. However, since $X$ commutes with $Y$, we know a priori that the periods in question should not depend
on the leaf. Hence they are all equal to zero.

Next recall that in the coordinate $T$, the restriction of $X$ to the leaf in question
becomes $(x_0^a y_0^b) v(x_0 e^{mT}, y_0 e^{-nT})  e^T \, \partial /\partial T$ which converges uniformly
to $x_0^a y_0^b e^T \, \partial /\partial T$ as both $x_0, y_0$ go to zero.
Since the shape of $\Omega \subset \C$
is determined by the foliation - and hence does not depend on multiplicative factors - we conclude
that the vector field $x_0^a y_0^b v(x_0 e^{mT}, y_0 e^{-nT}) e^T \, \partial /\partial T$ is semicomplete
from the fact that it is a perturbation - keeping all periods equal to zero - of the semicomplete vector field
corresponding to $v$ being constant equal to~$1$. Theorem~\ref{harder_cases} is proved.
\end{proof}

The remainder of the paper is devoted to the proof of Proposition~\ref{automaticlinearization}. Consider then vector fields
$X$ and $Y$ as in this proposition. If one of the foliations, say $\fol_X$, is linearizable. As in the proof
of Theorem~\ref{harder_cases}, there follows from  Lemmas~\ref{eliminatinginvertiblefactor-f} and~\ref{lemmaDependence1} that
 $X$ admits the normal form
$$
X = x^a y^b [mx \partial /\partial x - ny  \partial /\partial y]
$$
with $am-bn \in \{-1, 1\}$. The pair $X$ and $Y$ is then as indicated in Theorem~\ref{harder_cases} and there is nothing
to be proved. Thus, in order to prove Proposition~\ref{automaticlinearization}, we assume, aiming at a contradiction,
that neither $\fol_X$ nor $\fol_Y$ is linearizable. In particular, $\fol_X$ and $\fol_Y$ must have the same
eigenvalues $m, -n$ at the origin, with $m, n \in \N^{\ast}$ (see Corollary~\ref{corol_linearizable}).

Owing to Lemma~\ref{existence_firstcoordinates}, there exist coordinates $(x,y)$ where $X$ has the form
\begin{equation}
X =(x^ny^m)^{k_1} [mx (1+ {\rm h.o.t.}) \ddx - n y (1+ {\rm h.o.t.}) \ddy] \label{Form_final_for_X}
\end{equation}
while
\begin{equation}
Y = (x^ny^m)^{k_2} [mx (1+ {\rm h.o.t.}) \ddx - n y (1+ {\rm h.o.t.}) \ddy] \, , \label{Form_final_for_Y}
\end{equation}
where $k_1, k_2$ belong to $\N^{\ast}$. The following simple lemma will be useful in the discussion below.

\begin{lema}
\label{k_2greaterthank_1}
With the above notation, we can assume without loss of generality that $k_2 > k_1$ (strictly).
\end{lema}

\begin{proof}
The order of a germ of vector field is the degree of the first non-zero homogeneous component at the origin.
In particular, the order of $X$ (resp. $Y$) as above is $k_1+1$ (resp. $k_2 +1$).
Assume now that $k_1 =k_2$ otherwise there is nothing to be proved. Note however that the vector field
$Z = X-Y$ has order strictly greater than the orders of $X$ and of $Y$. Furthermore $X$ and $Z$ still form
a rank~$2$ systems of commuting semicomplete vector fields owing to Lemma~\ref{generalequivalence}. In particular
they satisfy conditions~(1) and~(2) at the beginning of this section. We can however assume that the foliation
$\fol_Z$ associated with $Z$ has a Martinet-Ramis singular point at the origin for, otherwise, the pair
$X$ and $Z$ falls back in one of the cases already treated of Main Theorem. Moreover, since by assumption $\fol_X$ is
not linearizable at the origin, there follows that $\fol_X$ and $\fol_Z$ share the same eigenvalues
(Corollary~\ref{corol_linearizable}). Thus we can replace $Y$ by $Z$ in the previous discussion and since
the order of $Z$ at the origin is strictly larger than the order of $X$, there will follow that $k_2 > k_1$ as desired.
\end{proof}

Since $X,\, Y$ are linearly independent at generic points (condition~(1)), the flow of $Y$ induces a non-trivial $1$-parameter
group of transformations on the leaf space of $\fol_X$. To exploit this observation,
it is convenient to briefly recall the description of the leaf space of $\fol_X$
provided by Martinet-Ramis in \cite{MR-ENS} (see also \cite{Frank}).

The starting point is to observe that $\fol_X$ admits a {\it formal}\, normal form corresponding
to the foliation $\fol_{\lambda, \tpp}$ defined by the $1$-form
$$
mx [ 1 + \lambda (x^n y^m)^\tpp] \, dy + n y [1 + (\lambda -1)(x^n y^m)^\tpp] \, dx \, .
$$
In particular, the complex number $\lambda$ and the positive integer $\tpp \geq 1$ are
the only formal invariants of $\fol_X$ ($m,n$ being fixed).

Let $h_X$ denote the local holonomy of $\fol_X$ with respect to the axis $\{ y=0\}$. Then $h_X$ is formally conjugate to $\sigma
\circ g_{m\tpp, \lambda}$, where $\sigma$ is the rotation $\sigma (z) = e^{2\pi i n/m} z$ and
$g_{m\tpp, \lambda}$ is the time-one map induced by the vector field
$$
Z_{m\tpp, \lambda} = 2 \pi i \frac{z^{m\tpp +1}}{1 + \lambda z^{m\tpp}}
\partz \, .
$$
The leaf space of $\fol_X$ on $\C^2 \setminus \{ x=0\}$ is identified
to the orbit space of $h_X$ on a neighborhood of $0 \in \C$. In
turn, the latter space can be realized as a ``$\tpp$-collar of
spheres'', cf. \cite{MR-ENS}. This is a $1$-dimensional non-separate
complex manifold $\Lambda$ with $2\tpp$ distinguished points $a_0, a_1,
\ldots , a_{2\tpp-1}$ satisfying the following condition: for every $i$
(mod $2\tpp$) $a_i, a_{i+1}$ belong to a unique Riemann sphere embedded
in $\Lambda$. Thus every $a_i$ belongs exactly to {\it two
consecutive spheres}\, glued together over neighborhoods of $a_i$ by
means of some analytic diffeomorphism. We choose one sphere and
denote it by $S_0^2$. The remaining spheres can then naturally be
ordered by the gluing points. Let $z_0$ be a coordinate on $S_0^2$
where the gluing points are identified to $0, \infty$. The next step
consists of determining a coordinate $z_1$ over $S_1^2$ such that
the gluing diffeomorphisms at the points $0, \infty$ are tangent to
the identity. The procedure is then repeated over each $S_i^2$ ($i=1, \ldots , 2\tpp-1$) by requiring
that the gluing diffeomorphisms
\begin{eqnarray*}
& & \varphi^j_{\infty} : (S_{2j+1} , \infty) \longrightarrow (S_{2j} ,
\infty) \;\; \; \; \; \; (j=0, \ldots , \tpp-1)  \; , \\
& & \varphi^j_{0} : (S_{2j} , 0) \longrightarrow (S_{2j-1} , 0)
\;\; \; \; \; \; (j=1, \ldots , \tpp-1)
\end{eqnarray*}
are tangent to the identity at the gluing points. In the coordinate
$z_0 , z_{2\tpp-1}$ the diffeomorphism $ \psi_0^0 (S_0^2 , 0)
\rightarrow (S_{2\tpp-1} ,0)$ becomes $\psi_0^0 = e^{-2\pi i \lambda}
\phi_0^0$ where $\phi_0^0$ is tangent to the identity. Summarizing
to each $\tpp$-collar of spheres, it is associated a pair $(\lambda ,
\varphi \in [{\rm Diff}\, (S^2; 0 , \infty )]^{\tpp})$. The arbitrary choices made
above are all encoded in the natural action of $\Z_{2\tpp} \times \Z_2
\times \C^{\ast}$ on the set of these pairs. Finally two $\tpp$-collar
of spheres are isomorphic if and only if their associated pairs
coincide modulo this action.

Through its sectorial normalizations, a local diffeomorphism
formally conjugate to $\sigma \circ g_{m\tpp, \lambda}$ induces a
$\tpp$-collar of spheres whose associated pair has the form $(\lambda
-qm/m^2, \varphi )$ for some $\varphi \in [{\rm Diff}\, (S^2; 0 ,
\infty )]^{\tpp}$ (where $q$ is such that $qn=1$ modulo $m$). Two
diffeomorphisms formally conjugate to $\sigma \circ g_{m\tpp, \lambda}$
are {\it analytically conjugate}\,to each other if, and only if, their
associated $\tpp$-collar of spheres are isomorphic. Alternatively these diffeomorphisms
are analytically conjugate if, and only if, their associated pairs
belong to the same orbit of the mentioned action of $\Z_{2\tpp} \times \Z_2 \times \C^{\ast}$. Finally,
the $\tpp$-collar of spheres corresponding to the diffeomorphism
$\sigma \circ g_{m\tpp, \lambda}$ is characterized by having $\varphi =
{\rm Id}$ (with slightly different conventions, this collar is
called euclidean in \cite{Frank}).

Let us now go back to the foliation $\fol_X$. It was seen that $Y$ induces
a non-trivial ($1$-parameter) group of automorphisms in the leaf space of $\fol_X$. By construction,
these are holomorphic automorphism of the (singular) Riemann surface $\Lambda$.
Since $\Lambda$ has dimension~$1$, it follows
that this group must coincide with the flow of a non-trivial
holomorphic vector field $\mathcal{Z}$ globally defined on
$\Lambda$. The existence of a non-trivial holomorphic vector field
clearly trivializes the invariant
$\varphi \in [{\rm Diff}\, (S^2; 0, \infty )]^\tpp$. Thus we obtain:

\begin{lema}
\label{Siegel1.com}
The foliation $\fol_X$ is analytically conjugate to the foliation given by
$$
mx [ 1 + \lambda (x^n y^m)^\tpp] \, dy + n y [1 + (\lambda -1)(x^n y^m)^\tpp] \, dx \, ,
$$
i.e. the formal conjugacy between $\fol_X$ and $\fol_{\lambda, \tpp}$ is, in fact,
analytic.\qed
\end{lema}

A (local) representative for the leaf space of $\fol_X$ can be obtained by means of a local transverse
section. More precisely consider a local transverse section $\Sigma$ to the leaves of
$\fol_X$ passing through a point $p_0 \in \{y=0\}$ so that $\Sigma$ can be identified to a neighborhood
of $0 \in \C$. The vector field $Y$ has a natural ``projection'' $Y_{\Sigma}$ on $\Sigma$ defined as follows.
Given local coordinates $(x_1, y_1)$ around $p_0$ so that $X=f(x_1,y_1) \partial /\partial x_1$,
the vector field $Y$ takes on the form $Y = F(x_1,y_1) \partial /\partial x_1 + G(y_1) \partial /\partial y_1$,
where $G$ depends only on~$y_1$ since $[X,Y]=0$.
The vector field $Y_{\Sigma}$ is then given by $Y_{\Sigma} = G(y_1) \partial /\partial y_1$.

\begin{proof}[Proof of Proposition~\ref{automaticlinearization}]
Keep the preceding notation and recall that we have assumed aiming at a contradiction
that neither $\fol_X$ nor $\fol_Y$ is linearizable. The vector fields $X$ and $Y$ are given
in suitable coordinates by Formulas~(\ref{Form_final_for_X}) and~(\ref{Form_final_for_Y}).
In view of Lemma~\ref{k_2greaterthank_1}, we also assume that $k_2 > k_1$.

Thanks to Lemma~\ref{Siegel1.com}, we may change the coordinates so that $X$ becomes
\begin{equation}
X= (x^ny^m)^{k_1} f(x,y) [
mx ( 1 + \lambda (x^n y^m)^\tpp) \partial /\partial x - n y (1 + (\lambda -1)(x^n y^m)^\tpp)
\partial /\partial y] \, , \label{formalnormalformX}
\end{equation}
where $f(0,0) \neq 1$ (up to a multiplicative constant). In the same coordinates $(x,y)$,
the first non-zero homogeneous component $Y^H$ of the vector field $Y$ is given by
\[
Y^H = (x^ny^m)^{k_2} [ mx \partial /\partial x - n y \partial /\partial y] \, .
\]
As already pointed out in the proof of Lemma~\ref{k_2greaterthank_1}, we can also assume that $Z = X-Y$
forms with $X$ a rank~$2$ system of commuting vector fields satisfying the conditions~(1), (2), and~(3)
at the beginning of the section. Moreover the first non-zero homogeneous component $Z^H$ of $Z$ is given by
\[
Z^H = (x^ny^m)^{k_1} [ mx \partial /\partial x - n y \partial /\partial y] \, .
\]
Thus the vector field $Z$ itself has the form
$$
Z = (x^ny^m)^{k_1} [ mx (1 + {\rm h.o.t.}) \partial /\partial x - n y (1 + {\rm h.o.t.}) \partial /\partial y] \, .
$$

Next let $\Sigma$ be a local section transverse to the leaves of $\fol_X$ and passing through a point $p_0
\in \{y=0\}$. The section $\Sigma$ will also be identified with a neighborhood of $0 \in \C$.
Consider local coordinates
$(x_1, y_1)$ about $p_0$ so that $\fol_X$ becomes horizontal, i.e. where $X$ becomes $X = f(x_1,y_1) \partial /\partial
x_1$. In these coordinates, the vector field $Y$ must admit the form
\[
Y = y_1^{mk} [a(x_1, y_1) \partial /\partial x_1 + y_1 b(y_1) \partial /\partial y_1]
\]
where $b(0) \neq 0$. Thus the projection $Y_{\Sigma}$ of $Y$ on $\Sigma$ is given in the coordinate $y_1$ by
$Y_{\Sigma} = y_1^{k_2m + 1} b(y_1) \partial /\partial y_1$.

Now note that the vector field $Y_{\Sigma}$ is preserved by the local holonomy map $h_X$
associated to the foliation $\fol_X$ w.r.t. to the leaf $\{ y=0\}$. Indeed this is an immediate
consequence of the fact that $Y$ has a natural projection on the leaf space of $\fol_X$.
However, it follows from Lemma~\ref{Siegel1.com} that $h_X$ is conjugate to the time-one map
induced by the vector field $2 \pi i (y_1^{m\tpp +1})/(1 + \lambda y_1^{m\tpp})
\partial /\partial y_1$. It is well known that the last condition guarantees that these
two vector fields must coincide up to a multiplicative constant.
In particular, we conclude that $k_2 = \tpp$.

The desired contradiction now arises from considering the projection $Z_{\Sigma}$ of $Z$ on $\Sigma$.
By repeating the previous argument, we are led to the conclusion that $k_1 = \tpp =k_2$ which is impossible
since $k_2 > k_1$ strictly. This ends the proof of Proposition~\ref{automaticlinearization}.
\end{proof}

\bigskip

\noindent {\bf Acknowledgements}.

The first author was partially supported by FCT (Portugal) through the sabbatical grant SFRH/BSAB/135549/2018.
The second author was partially supported by CIMI - Labex Toulouse - through the grant
``Dynamics of modular vector fields, Dwork family, and applications''. 
The third author was partially supported by CMUP (UID/MAT/00144/2013), which is funded by FCT (Portugal) with national (MEC)
and European structural funds through the programs FEDER, under the partnership agreement PT2020.
She was also partially supported by FCT through the sabbatical grant SFRH/BSAB/128143/2016. Finally all the
three authors benefited from CNRS (France) support through the PICS project ``Dynamics of Complex ODEs and Geometry''.


\bigskip

\bigskip

\begin{flushleft}
{\sc Ana Cristina Ferreira} \\
Centro de Matem\'atica da Universidade do Minho, \\
Campus de Gualtar, \\
4710-057 Braga, Portugal\\
anaferreira@math.uminho.pt \\

\end{flushleft}

\bigskip

\begin{flushleft}
{\sc Julio Rebelo} \\
Institut de Math\'ematiques de Toulouse ; UMR 5219\\
Universit\'e de Toulouse\\
118 Route de Narbonne\\
F-31062 Toulouse, FRANCE.\\
rebelo@math.univ-toulouse.fr

\end{flushleft}

\bigskip

\begin{flushleft}
{\sc Helena Reis} \\
Centro de Matem\'atica da Universidade do Porto, \\
Faculdade de Economia da Universidade do Porto, \\
Portugal\\
hreis@fep.up.pt \\

\end{flushleft}

\end{document}